\documentclass[preprint]{imsart}

\usepackage[latin1]{inputenc}

\RequirePackage{amsthm,amsmath,amssymb,amsfonts,calrsfs,natbib}
\RequirePackage[colorlinks,citecolor=blue,urlcolor=blue]{hyperref}
\RequirePackage{hypernat}

\usepackage{array}

\startlocaldefs
\numberwithin{equation}{section}

\newtheorem{proposition}{Proposition}[section]
\newtheorem{theorem}[proposition]{Theorem}
\newtheorem{lemma}[proposition]{Lemma}

\theoremstyle{remark}
\newtheorem{claim}{Claim}

\newtheorem{example}[proposition]{Example}

\renewcommand{\leq}{\leqslant}

\renewcommand{\le}{\leqslant}
\renewcommand{\ge}{\geqslant}
\newcommand{\eps}{\varepsilon}
\newcommand{\som}{{\textstyle\sum}}
\newcommand{\norm}[1]{\|#1\|}
\newcommand{\RV}{RV}
\renewcommand{\Pr}{\operatorname{P}}
\newcommand{\E}{\operatorname{E}}
\newcommand{\sign}{\operatorname{sign}}
\newcommand{\Id}{\operatorname{Id}}
\newcommand{\1}{\mathbf{1}}
\DeclareMathOperator\symdif{\triangle}
\newcommand{\dto}{\rightsquigarrow}
\newcommand{\eqd}{\stackrel{d}{=}}
\newcommand{\downto}{\downarrow}
\newcommand{\boundary}{\partial}

\newcommand{\defn}{\emph}

\newcommand{\Banach}{\mathbb{B}}
\newcommand{\B}{\Banach}
\newcommand{\dual}{\Banach^\ast}

\newcommand{\RR}{\mathbb{R}}
\newcommand{\ZZ}{{\mathbb{Z}}}
\newcommand{\Z}{\ZZ}
\newcommand{\NN}{{\mathbb{N}}}
\newcommand{\Czero}{{\cal C}_0}
\newcommand{\Bzero}{{\Banach_0}}
\newcommand{\ball}[1]{B_{0, {#1}}}
\newcommand{\Mzero}{M_0}
\newcommand{\sphere}{{\mathbb{S}}}
\newcommand{\law}{{\cal L}}
\newcommand{\Pareto}{\operatorname{Pareto}}

\endlocaldefs 

\begin{document}

\begin{frontmatter}
\title{Regularly varying time series in Banach spaces\protect\thanksref{}}
\runtitle{Regularly varying time series in Banach spaces}

\author{\fnms{Thomas} \snm{Meinguet}\thanksref{t1}\ead[label=e1]{thomas.meinguet@uclouvain.be}}
\and
\author{\fnms{Johan} \snm{Segers}\thanksref{t1}\corref{}\ead[label=e2]{johan.segers@uclouvain.be}}

\thankstext{t1}{Research supported by IAP research network grant nr.\ P6/03 of the Belgian government (Belgian Science Policy) and by contract nr.\ 07/12/002 of the Projet d'Actions de Recherche Concert\'ees of the Communaut\'e fran\c{c}aise de Belgique, granted by the Acad\'emie universitaire Louvain.}

\runauthor{T. Meinguet and J. Segers}

\affiliation{Universit\'{e} catholique de Louvain}

\address{Universit\'e catholique de Louvain, Institut de statistique\\
Voie du Roman Pays 20, B-1348 Louvain-la-Neuve, Belgium\\
\printead{e1,e2}}

\begin{abstract}
When a spatial process is recorded over time and the observation at a given time instant is viewed as a point in a function space, the result is a time series taking values in a Banach space. To study the spatio-temporal extremal dynamics of such a time series, the latter is assumed to be jointly regularly varying. This assumption is shown to be equivalent to convergence in distribution of the rescaled time series conditionally on the event that at a given moment in time it is far away from the origin. The limit is called the tail process or the spectral process depending on the way of rescaling. These processes provide convenient starting points to study, for instance, joint survival functions, tail dependence coefficients, extremograms, extremal indices, and point processes of extremes. The theory applies to linear processes composed of infinite sums of linearly transformed independent random elements whose common distribution is regularly varying.
\end{abstract}



\end{frontmatter}

\section{Introduction}
\label{S:intro}

A powerful way to model spatio-temporal phenomena is by means of time series of functional observations.  Objectives of the analysis include inference on and prediction of certain functionals of the process, for instance the integral or the maximum of the process over a certain subregion. For risk management purposes, the interest is often in the extremal dynamics of such processes, both within space and over time. Examples from the literature include sea levels along dikes in the Netherlands \citep{DHL01}, windspeeds along the faces of a building \citep{DM08}, and precipitation in the state of Colorado \citep{COO07}.

Standard functional data analysis starts from the assumption of finite second moments and proceeds via the sequences of mean and autocovariance operators of the process \citep{BOS00}. The time series models in use are mostly linear. In the heavy-tailed case, however, such an approach may be inadequate. Already in \citet{DR85}, sample covariance functions of univariate linear processes with regularly varying, infinite variance innovations were found to have nondegenerate stable limits. Similar results for nonlinear processes were obtained in \citet{DM98}.

While originally defined for univariate functions and random variables, the concept of regular variation has by now been extended to quite abstract settings, including the one of stochastic processes \citep{HL05,HL06}. In \citet{DM08}, the extreme-value behavior of a certain linear process with regularly varying innovations in Skorohod space was investigated. As for random variables, regular variation provides the mathematical backbone for a coherent theory of extreme values of random functions. By considering the functional observations as points in a suitable function space, we are led to consider regularly varying time series taking values in a Banach space. 

Our aim is to find a convenient way to express and study interesting tail-related quantities of time series of functional data, such as joint survival functions, tail dependence coefficients, extremograms \citep{DM09}, extremal indices and other characteristics of clusters of extremes. Moreover we aim at specializing these results to linear processes with regularly varying innovations.

As in the finite-dimensional case \citep{BS09}, \defn{(joint) regular variation} of a stationary time series $(X_t)_{t \in \mathbb{Z}}$ in a separable Banach space $\Banach$ is shown to be equivalent to the existence of the limit in distribution of the rescaled process
\[
  (X_t/u)_{t \in \mathbb{Z}} \text{ conditionally on } \norm{X_0} > u \text{ as } u \to \infty
\]
in the proper product space. The limit in distribution, denoted by $(Y_t)_{t\in\Z}$, is the \defn{tail process}. It admits a familiar-looking decomposition into independent radial and angular components. The radial component is fully determined by the index of regular variation $\alpha$ of the random variable $\norm{X_0}$, while the angular component, called the \defn{spectral process}, effectively captures all aspects of extremal dependence, both within space and over time. Specifically, the spectral process $(\Theta_t)_{t\in\Z}$ is given by the limit in distribution of
\[
  (X_t/\norm{X_0})_{t \in \mathbb{Z}} \text{ conditionally on } \norm{X_0} > u \text{ as } u \to \infty.
\]
The distributions of the tail and spectral processes are uniquely determined by the ones of their restrictions to the nonnegative time axis.

The property of regular variation turns out to be preserved under bounded linear operators. This allows us to study \defn{linear processes}
\[
  X_t = \sum_{i \in \ZZ} T_i Z_{t-i}, \qquad t \in \ZZ,
\]
with $(Z_t)_{t \in \ZZ}$ independent random elements in a Banach space $\Banach_1$ and $T_i$ ($i \in \ZZ$) bounded linear operators from $\Banach_1$ to a second Banach space $\Banach_2$. If the common distribution of the innovations $Z_t$ is regularly varying and if the norms of the operators $T_i$ satisfy the summability condition \eqref{E:Resnick} below, the series $(X_t)_{t \in \ZZ}$ is regularly varying in $\Banach_2$. Its spectral process reflects the common extreme-value heuristic that large values in the series most likely arise from a single large shock among the innovations.

All in all we find that the spectral process provides a convenient, natural, and unifying concept for expressing and studying extremal characteristics of regularly varying time series in infinite-dimensional spaces. A next step could be to reevaluate certain methods from functional data analysis when basic moment assumptions are violated and are replaced by the assumption of regular variation. Finally, note that in some of our results, the vector space structure of the Banach space $\Banach$ does not come into play, so that certain parts of the theory might even be carried over to more general spaces such as cones \citep{HL06, DMZ08}.

The paper is organized in two parts. In the first part, Sections~\ref{S:RV}--\ref{S:uses}, the theory of regular variation of random elements and stationary time series in real, separable Banach spaces is developed in general. The emphasis is on the properties and the use of the spectral process. In the second part, Section~\ref{S:linear}--\ref{S:Examples}, the effect of bounded linear operators on regular variation is investigated, with applications to infinite random sums and linear processes with regularly varying innovations. Some auxiliary results are relegated to the Appendix.

\section{Regular variation}
\label{S:RV}

Regular variation of probability measures on Euclidean space is usually defined in terms of vague convergence of a sequence of Radon measures, living on Euclidean space compactified at infinity and punctured at the origin \citep{RES07}. For Banach spaces $\Banach$ that are not locally compact, such an approach does not work. A possible way out is to replace vague convergence by $\hat{w}$-convergence described in \cite{DVJ88}. This, however, requires changing the metric on $\Banach$ and completing it at infinity \citep{DHL01, HL05, DHF06, DM08}. These steps are not needed in the approach of \cite{HL06} based on $M_0$-zero convergence, which we follow here. In addition, we provide a number of characterizations of regular variation that are purely probabilistic.

Let $\Banach$ be a real, separable Banach space. For $r > 0$, let $B_{0, u} = \{ x \in \Banach : \norm{x} < u \}$ be the open ball in $\Banach$ centered at $0$ with radius $u$. Let $\Mzero = \Mzero(\Banach)$ be the class of Borel measures on $\Bzero = \Banach \setminus \{ 0 \}$ whose restriction to $\Banach \setminus \ball{u}$ is finite for all $u > 0$. Put $\Czero = \Czero(\Banach)$ the class of bounded and continuous functions $f : \Bzero \to \RR$ for which there exists $u > 0$ such that $f$ vanishes on $B_{0, u}$. A sequence of measures $\mu_n$ in $\Mzero$ converges to $\mu$ in $\Mzero$ if and only if $\lim_{n \to \infty} \int f \, d\mu_n = \int f \, d\mu$ for all $f \in \Czero$. Essentially, convergence in $\Mzero$ is equivalent to weak convergence of finite measures when restricted to $\Banach \setminus \ball{u}$ for all but at most countably many $u > 0$. The topology of $\Mzero$-convergence is metrizable, turning $\Mzero$ into a complete, separable metric space. Versions of the Portmanteau theorem, the continuous mapping theorem, and Prohorov's theorem hold; see \citet[Section~2]{HL06}. Finally, let $\RV_\tau$ be the class of functions that are regularly varying at infinity with index $\tau \in \RR$, that is, the class of measurable functions $f : (0, \infty) \to (0, \infty)$ such that $\lim_{u \to \infty} f(ux)/f(u) = x^\tau$ for all $x > 0$.

A random element $X$ in $\Banach$ is \defn{regularly varying with index $\alpha > 0$} if there exists a nonzero $\mu \in \Mzero$ and a function $V \in \RV_{-\alpha}$ such that
\begin{equation}
\label{E:RV}
  \frac{1}{V(u)} \Pr(u^{-1}X \in \,\cdot\,) \to \mu \qquad (u \to \infty) \quad \text{in $\Mzero$}.
\end{equation}
The limit measure $\mu$ is defined up to a multiplicative constant only and it satisfies the homogeneity property $\mu(r A) = r^{-\alpha} \, \mu(A)$ for every $r > 0$ and Borel set $A \subset \Bzero$. Let $\sphere = \{ \theta \in \Banach : \norm{\theta} = 1 \}$ denote the unit sphere in $\Banach$. By homogeneity, $\mu(\sphere) = 0$. Equation~\eqref{E:RV} and the Portmanteau theorem then imply 
\[
  \Pr(\norm{X} > u) / V(u) \to \mu ( \{ x \in \Banach : \norm{x} > 1 \} ) =: c \ne 0 \qquad (u \to \infty).
\]

For $\mu$ as in \eqref{E:RV}, define a probability measure $\lambda$ on $\sphere$ by
\[
  \lambda(A) = c^{-1} \, \mu( \{ x \in \Banach : \norm{x} > 1, \, x / \norm{x} \in A \} ),
  \qquad \text{Borel sets $A \subset \sphere$}.
\]
We call $\lambda$ the \defn{spectral (probability) measure} of $X$. Put 
\begin{equation}
\label{E:polar}
  T : (0, \infty) \times \sphere \to \Bzero : (r, \theta) \mapsto r \theta
\end{equation}
and let $\nu_\alpha$ be the measure on $(0, \infty)$ given by $\nu_\alpha(dr) = d(-r^{-\alpha}) = \alpha \, r^{-\alpha-1} \, dr$ for $r > 0$. We have
\begin{equation}
\label{E:lambda2mu}
  \mu = c \, (\nu_\alpha \otimes \lambda) \circ T^{-1},
\end{equation}
with `$\otimes$' denoting product measure, an expression which is equivalent to
\begin{equation}
\label{E:lambda2mu:f}
  \int_\Bzero f \, d\mu = c \int_{\sphere} \int_0^\infty f(r\theta) \, d(-r^{-\alpha}) \, \lambda(d\theta)
\end{equation}
for all $\mu$-integrable $f : \Bzero \to \RR$.

In the following proposition, we provide three characterizations of regular variation in terms of weak convergence of probability distributions. Let $\Pareto(\alpha)$ denote the Pareto distribution with parameter $\alpha$, that is, the distribution of a positive random variable $Y$ with tail function $\Pr(Y > y) = y^{-\alpha}$ for $y \ge 1$. By $\law(Z)$ we mean the law of a random element $Z$, and $\law(Z \mid A)$ denotes the law of $Z$ conditionally on the event $A$. Let the arrow `$\dto$' denote convergence in distribution and let $\1(A)$ denote the indicator variable of the event $A$.

\begin{proposition}
\label{P:RV}
Let $X$ be a random element in $\Banach$, let $\alpha > 0$, and let $\lambda$ be a probability measure on the unit sphere $\sphere$ in $\Banach$. The following statements are equivalent:
\begin{itemize}
\item[(i)] $X$ is regularly varying with index $\alpha > 0$ and spectral measure $\lambda$.
\item[(ii)] The function $u \mapsto \Pr(\norm{X} > u)$ belongs to $\RV_{-\alpha}$ and
\[
  \law( X / \norm{X} \mid \norm{X} > u ) \dto \lambda \qquad (u \to \infty), \quad \text{in $\sphere$}.
\]
\item[(iii)] In $(0, \infty) \times \sphere$, we have
\[
  \law \bigl( \norm{X} / u, \, X / \norm{X} \, \big| \, \norm{X} > u \bigr) \dto \Pareto(\alpha) \otimes \lambda \qquad (u \to \infty).
\]
\item[(iv)] In $\Banach$, we have, with $T$ as in \eqref{E:polar},
\[
  \law( X / u \mid \norm{X} > u) \dto \bigl(\Pareto(\alpha) \otimes \lambda\bigr) \circ T^{-1} \qquad (u \to \infty).
\]
\end{itemize}
\end{proposition}

\begin{proof}
\textit{(i) implies (ii).}  As explained in the paragraph following $\eqref{E:RV}$, a possible choice for $V$ in \eqref{E:RV} is $V(u) = \Pr(\norm{X} > u)$, in which case $\mu(\{ x \in \Banach : \norm{x} > 1 \}) = 1$. Let $g : \sphere \to \RR$ be bounded and continuous and put $f(x) = g(x/\norm{x}) \, \1(\norm{x} > 1)$ for $x \in \Bzero$. The discontinuity set of $f$ is contained in $\sphere$, which is a $\mu$-null set. Hence by Lemma~\ref{L:CMT} and equation \eqref{E:lambda2mu:f}, as $u \to \infty$,
\[
  \E [ g(X / \norm{X}) \mid \norm{X} > u ] = \frac{1}{V(u)} \E [ f(X/u) ] 
  \to \int_{\Bzero} f \, d\mu = \int_{\sphere} g \, d\lambda.
\]

\textit{(ii) implies (iii).} Let $y \ge 1$ and let $g : \sphere \to \RR$ be bounded and continuous. By (ii), as $n \to \infty$,
\begin{multline*}
  \E [ \1(\norm{X}/u > y) \, g(X / \norm{X}) \mid \norm{X} > u ] \\
  = \frac{\Pr(\norm{X} > uy)}{\Pr(\norm{X} > u)} \, \E [ g(X / \norm{X}) \mid \norm{X} > uy ]
  \to y^{-\alpha} \, \int_\sphere g(\theta) \, \lambda(d\theta).
\end{multline*}
In view of Lemma~\ref{L:product}, this implies (iii).

\textit{(iii) implies (i).} Let $f \in \Czero$. Let $z > 0$ be such that $f$ vanishes on $B_{0,z}$. Put $V(u) = \Pr(\norm{X} > u)$. The weak convergence relation $\law( \norm{X} / u \mid \norm{X} > u ) \dto \Pareto(\alpha)$ as $u \to \infty$ implies $V \in \RV_{-\alpha}$. Let $(Y, \Theta)$ be a random element in $(0, \infty) \times \sphere$ with distribution $\Pareto(\alpha) \otimes \lambda$. We have
\begin{align*}
  \frac{1}{V(u)} \E [ f(X/u) ] 
  &= \frac{1}{V(u)} \E [ f(X/u) \, \1(\norm{X} > uz) ] \\
  &= \frac{V(uz)}{V(u)} \E \biggl[ f \biggl( z \frac{\norm{X}}{uz} \frac{X}{\norm{X}} \biggr) \, \bigg| \, \norm{X} > uz \biggr]
  \to z^{-\alpha} \E [ f(zY\Theta) ]
\end{align*}
as $u \to \infty$. By Fubini's theorem, the limit is equal to
\begin{align*}
  z^{-\alpha} \int_\sphere \int_1^\infty f(zy\theta) \, d(-y^{-\alpha}) \, \lambda(d\theta)
  &= \int_\sphere \int_z^\infty f(r\theta) \, d(-r^{-\alpha}) \, \lambda(d\theta) \\
  &= \int_\sphere \int_0^\infty f(r\theta) \, d(-r^{-\alpha}) \, \lambda(d\theta).
\end{align*}
We obtain \eqref{E:RV} with $\mu$ as in \eqref{E:lambda2mu}.

\textit{(iii) implies (iv)}, and \textit{(iv) implies (ii).} By applications of the continuous mapping theorem, upon noting that 
\[
  \frac{X}{u} =  \frac{\norm{X}}{u} \frac{X}{\norm{X}} \qquad \text{and} \qquad \frac{X}{\norm{X}} = \frac{X/u}{\norm{X/u}}. \qedhere
\]
\end{proof}

\section{Spectral process of a stationary time series}
\label{S:timeseries}

Let $(X_t)_{t \in \ZZ}$ be a (strictly) stationary time series in a (real) separable Banach space $\Banach$. The time series is said to be \defn{(jointly) regularly varying with index $\alpha > 0$} if for every positive integer $k$ the vector $(X_1, \ldots, X_k)$ is regularly varying with index $\alpha$ in the Banach space $\Banach^k$. One possible choice for the norm on $\Banach^k$ is $\norm{(x_1, \ldots, x_k)} = \max(\norm{x_1}, \ldots, \norm{x_k})$.

According to the definition above, there exist sequences of measures $\mu_k$ in $M_0(\Banach^k)$ and functions $V_k \in \RV_{-\alpha}$ such that for every positive integer $k$,
\[
  \frac{1}{V_k(u)} \Pr [(u^{-1}X_1, \ldots, u^{-1}X_k) \in \, \cdot \,] \to \mu_k \qquad (u \to \infty) \quad \text{in $M_0(\Banach^k)$}.
\]
These limiting measures give rise to spectral measures $\lambda_k$ on $\sphere_k$, the unit sphere in $\Banach^k$. Although these spectral measures originate from a single, stationary process, it is awkward to describe how they are related because the spheres on which they live are of different dimensions. Therefore, we seek an alternative description in terms of a single object.

In the following, $\Banach^\ZZ$ denotes the space of sequences $(x_t)_{t \in \ZZ}$ in $\Banach$ endowed with the product topology, that is, the topology of elementwise convergence. In $\Banach^\ZZ$, convergence in distribution is equivalent to convergence in distribution in $\Banach^k$ of all finite stretches of length $k$, for every positive integer $k$ (convergence of finite-dimensional distributions, so to speak, except that the dimension of a single $\Banach$ may already be infinite). By convention, $\max \varnothing = 0$.

\begin{theorem}[Spectral process]
\label{T:spectral}
Let $(X_t)_{t \in \ZZ}$ be a stationary time series in $\Banach$ and let $\alpha > 0$. The following statements are equivalent:
\begin{itemize}
\item[(i)] $(X_t)_{t \in \ZZ}$ is regularly varying with index $\alpha$.
\item[(ii)] The function $u \mapsto \Pr(\norm{X_0} > u)$ belongs to $\RV_{-\alpha}$ and here exists a random element $(\Theta_t)_{t \in \ZZ}$ in $\Banach^\ZZ$ such that
\[
  \law \bigl( (X_t / \norm{X_0})_{t \in \ZZ} \, \big| \, \norm{X_0} > u \bigr) \dto (\Theta_t)_{t \in \ZZ} \qquad (u \to \infty).
\]
\item[(iii)] There exists a random element $(\Theta_t)_{t \in \ZZ}$ in $\Banach^\ZZ$ such that in $(0, \infty) \times \Banach^\ZZ$,
\[
  \law \bigl( \norm{X_0} / u, (X_t / \norm{X_0})_{t \in \ZZ} \, \big| \, \norm{X_0} > u \bigr) \dto \bigl( Y, (\Theta_t)_{t \in \ZZ} \bigr)
  \qquad (u \to \infty),
\]
where $Y$ is a $\Pareto(\alpha)$ random variable independent from $(\Theta_t)_{t \in \ZZ}$.
\item[(iv)] There exists a random element $(\Theta_t)_{t \in \ZZ}$ in $\Banach^\ZZ$ such that
\[
  \law\bigl( (X_t / u)_{t \in \ZZ} \mid \norm{X_0} > u \bigr) \dto (Y \Theta_t)_{t \in \ZZ} \qquad (u \to \infty),
\]
where $Y$ is a $\Pareto(\alpha)$ random variable independent from $(\Theta_t)_{t \in \ZZ}$.
\end{itemize}
In this case, the object $(\Theta_t)_{t \in \ZZ}$ is the same across (ii)--(iv) and for every positive integer $k$,
\begin{equation}
\label{E:Theta2muk}
  \frac{1}{\Pr(\norm{X_0} > u)} \Pr[(X_1/u, \ldots, X_k/u) \in \,\cdot\,] \to \mu_k \qquad (u \to \infty)
\end{equation}
in $\Mzero(\Banach^k)$, where $\int f d\mu_k$ for $f \in \Czero(\Banach^k)$ is given by
\[  
  \sum_{j=1}^k \int_0^\infty \E \biggl[ f(0, \ldots, 0, r\Theta_0, \ldots, r\Theta_{k-j}) \, \1 \biggl( \max_{-j+1 \le i \le -1} \norm{\Theta_i} = 0 \biggr) \biggr] \, d(-r^{-\alpha}).
\]
\end{theorem}

The limit process $(\Theta_t)_{t \in \ZZ}$ is called the \defn{spectral process} of $(X_t)_{t \in \ZZ}$. The spectral process provides a unifying concept for a great variety of tail-related objects, see the examples of its use in Section~\ref{S:uses}.

Note that $\Theta_0$ is an $\sphere$-valued random element with law equal to the spectral measure of the common distribution of the random elements $X_t$. The process
\begin{equation}
\label{E:TailProcess}
  (Y_t)_{t \in \ZZ} = (Y \Theta_t)_{t \in \ZZ}
\end{equation}
in item~(iv) is called the \defn{tail process} of $(X_t)_{t \in \ZZ}$. Since $\norm{\Theta_0} = 1$, we have $Y = \norm{Y_0}$ and $\Theta_t = Y_t / \norm{Y_0}$ for all $t \in \ZZ$.

\begin{proof}[Proof of Theorem~\ref{T:spectral}]
\textit{(i) implies (ii).} By stationarity of $(X_t)_{t \in \ZZ}$ and regular variation of $(X_1, \ldots, X_k)$ in $\Banach^k$, it is not difficult to see that the limit
\[
  \lim_{u \to \infty} \frac{\Pr(\norm{X_0} > u)}{\Pr[ \max( \norm{X_1}, \ldots, \norm{X_k} ) ]}
\]
exists and is in $[1/k, 1]$ for every integer $k \ge 1$. As a consequence, a valid choice for the auxiliary function in the definition of regular variation of $(X_1, \ldots, X_k)$ is just $u \mapsto \Pr(\norm{X_0} > u)$, independently of $k$, yielding
\begin{equation}
\label{E:muk}
  \frac{1}{\Pr(\norm{X_0} > u)} \, \Pr[ (X_1/u, \ldots, X_k/u) \in \, \cdot \,] \to \mu_k \qquad (u \to \infty)
\end{equation}
in $\Mzero(\Banach^k)$. With this normalization, the limit measure $\mu_k$ satisfies
\[
  \mu_k( \{ (x_1, \ldots, x_k) \in \Banach^k : \norm{x_j} > 1 \} ) = 1, \qquad j \in \{1, \ldots, k\}.
\]

Let $s, t$ be nonnegative integers and write $k = t+s+1$. Put 
\begin{equation}
\label{E:sphere:st}
  \sphere_{s, t} = \{ (\theta_{-s}, \ldots, \theta_t) \in \Banach^k : \norm{\theta_0} = 1 \}
\end{equation}
and define a probability measure $\lambda_{s,t}$ on $\sphere_{s,t}$ by
\begin{multline}
\label{E:must2lambdast}
  \lambda_{s, t}(B) = \mu_k( \{ (x_{-s}, \ldots, x_t) \in \Banach^k : \\
  \norm{x_0} > 1, \, (x_{-s}/\norm{x_0}, \ldots, x_t/\norm{x_0}) \in B \} )
\end{multline}
for Borel subsets $B \subset \sphere_{s,t}$. Let $g : \sphere_{s, t} \to \RR$ be bounded and continuous and define $f : \Banach^k \to \RR$ by 
\[
  f(x_{-s}, \ldots, x_t) = g(x_{-s} / \norm{x_0}, \ldots, x_t / \norm{x_0}) \, \1 (\norm{x_0} > 1) 
\]
to be interpreted as $0$ if $x_0 = 0$. The function $f$ is bounded, vanishes on the unit ball in $\Banach^k$, and is continuous everywhere except perhaps on $\sphere_{s,t}$, which is a $\mu_k$-null set. As a consequence, by Lemma~\ref{L:CMT},
\begin{multline*}
  \E [ g(X_{-s} / \norm{X_0}, \ldots, X_t / \norm{X_0}) \mid \norm{X_0} > u ] \\
  = \frac{1}{\Pr(\norm{X_0} > u)} \E [ f(X_{-s}/u, \ldots, X_t/u) ] \\
  \to \int_{\Banach^k} f \, d\mu_k = \int_{\sphere_{s,t}} g \, d\lambda_{s,t} \qquad (u \to \infty).
\end{multline*}
That is, if $(\Theta_{-s}, \ldots, \Theta_t)$ is a random element of $\sphere_{s,t}$ with distribution $\lambda_{s,t}$, then
\[
  \law( X_{-s} / \norm{X_0}, \ldots, X_t / \norm{X_0} \mid \norm{X_0} > u )
  \dto (\Theta_{-s}, \ldots, \Theta_t), \qquad (u \to \infty).
\]
By the Daniell--Kolmogorov extension theorem \citep[Chapter~4, Theorem~53]{POL02}, there exists a random element $(\Theta_t)_{t \in \ZZ}$ in $\Banach^\ZZ$ such that the distribution of $(\Theta_{-s}, \ldots, \Theta_t)$ is $\lambda_{s,t}$ for all nonnegative integers $s$ and $t$. Weak convergence of finite stretches characterizing weak convergence in the product space $\Banach^\ZZ$ \citep[Theorem~1.4.8]{VDV96}, statement~(ii) follows.

\textit{(ii) implies (iii).} Let $s$ and $t$ be nonnegative integers, let $y \ge 1$ and let $g : \sphere_{s,t} \to \RR$ be continuous and bounded, with $\sphere_{s,t}$ as in \eqref{E:sphere:st}. We have
\begin{multline*}
  \E [ \1 ( \norm{X_0} / u > y ) \, g(X_{-s} / \norm{X_0}, \ldots, X_t / \norm{X_0}) \mid \norm{X_0} > u ] \\
  = \frac{\Pr(\norm{X_0} > uy)}{\Pr(\norm{X_0} > u)} \,
  \E [ g(X_{-s} / \norm{X_0}, \ldots, X_t / \norm{X_0}) \mid \norm{X_0} > uy ] \\
  \to y^{-\alpha} \, \E [ g(\Theta_{-s}, \ldots, \Theta_t) ] \qquad (u \to \infty).
\end{multline*}
In view of Lemma~\ref{L:product}, we find, as $u \to \infty$,
\[
  \law \bigl( \norm{X_0} / u, X_{-s} / \norm{X_0}, \ldots, X_t / \norm{X_0} \mid \norm{X_0} > u \bigr)
  \dto \bigl(Y, \Theta_{-s}, \ldots, \Theta_t),
\]
with $Y$ a $\Pareto(\alpha)$ random variable independent of $(\Theta_{-s}, \ldots, \Theta_t)$. Statement (iii) follows.

\textit{(iii) implies (i) and \eqref{E:Theta2muk}.} Let $k$ be a positive integer, and let $f \in \Czero(\Banach^k)$. There exists $z > 0$ such that $f$ vanishes on the ball $B_{0,z}$ in $\Banach^k$, that is, $f(x_1, \ldots, x_k) = 0$ whenever $\norm{x_j} < z$ for all $j \in \{1, \ldots, k\}$. Put $V(u) = \Pr(\norm{X_0} > u)$, a function which by (iii) is in $\RV_{-\alpha}$. Decomposing the event $\{ \max_{1 \le j \le k} \norm{X_j} > uz \}$ according to the smallest $j$ such that $\norm{X_j} > uz$, we find
\begin{align*}
  \lefteqn{
  \frac{1}{V(u)} \E [ f(X_1/u, \ldots, X_k/u) ] 
  } \\
  &= \frac{1}{V(u)} \sum_{j=1}^k \E \biggl[ f(X_1/u, \ldots, X_k/u) \, \1 \biggl( \norm{X_j} > uz \ge \max_{1 \le i \le j-1} \norm{X_i} \biggr) \biggr] \\
  &= \frac{V(uz)}{V(u)} \sum_{j=1}^k 
    \E \biggl[ f(X_1/u, \ldots, X_k/u) \, \1 \biggl( \max_{1 \le i \le j-1} \norm{X_i} < uz \biggr) \, \bigg| \, \norm{X_j} > uz \biggr].
\end{align*}
By stationarity, the terms in the final sum in the above display are equal to
\begin{align*}
  \E \biggl[ f(X_{1-j}/u, \ldots, X_{k-j}/u) \, \1 \biggl( \max_{-j+1 \le i \le -1} \norm{X_i} < uz \biggr) \, \bigg| \, \norm{X_0} > uz \biggr]
\end{align*}
for $j \in \{1, \ldots, k\}$. Writing 
\[
  \frac{X_i}{u} = z \, \frac{\norm{X_0}}{uz} \, \frac{X_i}{\norm{X_0}},
\]
we find, by (iii) and by continuity of the $\Pareto(\alpha)$ distribution,
\begin{multline*}
  \lim_{u \to \infty} \frac{1}{V(u)} \E [ f(X_1/u, \ldots, X_k/u) ] = \\
  z^{-\alpha} \sum_{j=1}^k \int_1^\infty \E \biggl[ f(zy\Theta_{1-j}, \ldots, zy\Theta_{k-j}) \1 \biggl( \max_{-j+1 \le i \le -1} \norm{y \Theta_i} < 1 \biggr) \biggr]
  d(-y^{-\alpha}).
\end{multline*}
The substitution $r = zy$ and the fact that $f$ vanishes on the ball $B_{0,z}$ in $\Banach^k$ yield
\begin{align*}
  \lefteqn{
  \lim_{u \to \infty} \frac{1}{V(u)} \E [ f(X_1/u, \ldots, X_k/u) ] 
  } \\
  &= \sum_{j=1}^k \int_z^\infty \E \biggl[ f(r\Theta_{1-j}, \ldots, r\Theta_{k-j}) \, \1 \biggl( \max_{-j+1 \le i \le -1} \norm{r \Theta_i} < z \biggr) \biggr] \, d(-r^{-\alpha}) \\
  &= \sum_{j=1}^k \int_0^\infty \E \biggl[ f(r\Theta_{1-j}, \ldots, r\Theta_{k-j}) \, \1 \biggl( \max_{-j+1 \le i \le -1} \norm{r \Theta_i} < z \biggr) \biggr] \, d(-r^{-\alpha}).
\end{align*}
Since this is true for all positive $z$ in a neighbourhood of zero, we obtain, by dominated convergence,
\begin{multline*}
  \lim_{u \to \infty} \frac{1}{V(u)} \E [ f(X_1/u, \ldots, X_k/u) ] = \\
  \sum_{j=1}^k \int_0^\infty \E \biggl[ f(0, \ldots, 0, r\Theta_0, \ldots, r\Theta_{k-j}) \, \1 \biggl( \max_{-j+1 \le i \le -1} \norm{\Theta_i} = 0 \biggr) \biggr] \, d(-r^{-\alpha}).
\end{multline*}
We obtain (iii) as well as \eqref{E:Theta2muk}. 

\textit{(iii) implies (iv)}, and \textit{(iv) implies (ii).} By applications of the continuous mapping theorem.
\end{proof}

\begin{example}[Asymptotic independence]
Assume that the common distribution of the random elements $X_t$ is regularly varying with spectral measure $\lambda$ and that $\norm{X_0}$ and $\norm{X_t}$ are asymptotically independent for each nonzero integer $t$ in the sense that
\[
  \lim_{u \to \infty} \Pr( \norm{X_t} > u \mid \norm{X_0} > u ) = 0.
\]
Then the time series $(X_t)_{t \in \ZZ}$ is regularly varying with spectral process $(\Theta_t)_{t \in \ZZ}$ with $\law(\Theta_0) = \lambda$ and $\Theta_t = 0$ almost surely for every $t \in \ZZ \setminus \{0\}$.
\end{example}

\section{The time-change formula}
\label{S:TimeChange}

In general, the spectral process $(\Theta_t)_{t \in \ZZ}$ of a stationary regularly varying time series $(X_t)_{t \in \ZZ}$ is itself nonstationary. Still, the fact that $(X_t)_{t \in \ZZ}$ is stationary induces a peculiar structure on the distribution of the spectral process. In particular, the distribution of $(\Theta_t)_{t \in \ZZ}$ is determined by the distribution of its restriction to the nonnegative time axis, that is, of the \defn{forward spectral process} $(\Theta_t)_{t \in \ZZ_+}$, with $\ZZ_+ = \{0, 1, 2, \ldots \}$. The same is true for the \defn{backward spectral process} $(\Theta_t)_{t \in \ZZ_-}$, with $\ZZ_- = \{0, -1, -2, \ldots \}$.

\begin{theorem}
\label{T:time}
Statements (ii)--(iv) in Theorem~\ref{T:spectral} are equivalent to the same statements with $\ZZ$ replaced by $\ZZ_+$ or $\ZZ_-$. In that case, 
\begin{equation}
\label{E:time}
  \E [ f(\Theta_{-s}, \ldots, \Theta_t) ]
  = \E \biggl[ f \biggl( \frac{\Theta_0}{\norm{\Theta_s}}, \ldots, \frac{\Theta_{t+s}}{\norm{\Theta_s}} \biggr) \, \norm{\Theta_s}^\alpha \biggr]
\end{equation}
for all nonnegative integer $s$ and $t$ and for all integrable functions $f : \Banach^{t+s+1} \to \RR$ that have the property that $f(\theta_{-s}, \ldots, \theta_t) = 0$ whenever $\theta_{-s} = 0$.
\end{theorem}

The proof of Theorem~\ref{T:time} is given below. Note that by considering the time-reversed process $\tilde{X}_t = X_{-t}$, equation~\eqref{E:time} can be reversed in the obvious way.

Some examples of the time-change formula \eqref{E:time} are given in Examples~\ref{ex:JointSurvFunc} and~\ref{ex:TailDependence}. A simple case occurs when $f$ only depends on its first component, that is, when $f(\theta_{-s}, \ldots, \theta_t) = f(\theta_{-s})$ and $f(0) = 0$: equation~\eqref{E:time} then reduces to
\begin{equation}
\label{E:time:s}
  \E [ f(\Theta_{-s}) ] = E [ f ( \Theta_0 / \norm{\Theta_s} ) \, \norm{\Theta_s}^\alpha ], \qquad s \in \ZZ.
\end{equation}
This yields an expression of the distribution of $\Theta_{-s}$ in terms of the joint law of $\Theta_0$ and $\Theta_s$. In particular we find
\[
  \Pr( \Theta_{-s} \ne 0 ) = \E [ \norm{\Theta_s}^\alpha ], \qquad s \in \ZZ.
\]
If the common value in the preceding display is equal to unity, then \eqref{E:time:s} is valid for arbitrary integrable $f$, that is, without the restriction that $f(0) = 0$. 

\begin{proof}[Proof of Theorem~\ref{T:time}]
By symmetry, we only need to consider the forward case, $\ZZ_+ = \{0, 1, 2, \ldots\}$. Consider the statements~(ii) and~(iii) in Theorem~\ref{T:spectral} with $\ZZ$ replaced by $\ZZ_+$:
\begin{itemize}
\item[(ii$_+$)] The function $u \mapsto \Pr(\norm{X_0} > u)$ belongs to $\RV_{-\alpha}$ and in $\Banach^{\ZZ_+}$,
\[
  \law \bigl( (X_t / \norm{X_0})_{t \in \ZZ_+} \, \big| \, \norm{X_0} > u \bigr) \dto (\Theta_t)_{t \in \ZZ_+} \qquad (u \to \infty).
\]
\item[(iii$_+$)] In $(0, \infty) \times \Banach^{\ZZ_+}$, as $u \to \infty$,
\[
  \law \bigl( \norm{X_0} / u, (X_t / \norm{X_0})_{t \in \ZZ_+} \, \big| \, \norm{X_0} > u \bigr) \dto \bigl( Y, (\Theta_t)_{t \in \ZZ_+} \bigr),
\]
where $Y$ is a $\Pareto(\alpha)$ random variable independent from $(\Theta_t)_{t \in \ZZ_+}$.
\end{itemize}
We have to show that the statements (i)--(iv) in Theorem~\ref{T:spectral} are equivalent with each of (ii$_+$) and (iii$_+$). We already know that (i) implies (ii). Trivially, (ii) implies (ii$_+$). To show that (ii$_+$) implies (iii$_+$), just set $s = 0$ in the part of the proof of Theorem~\ref{T:spectral} that (ii) implies (iii). Since (iii) implies (i) by Theorem~\ref{T:spectral}, all that remains to be shown is that (iii$_+$) implies (iii). As before, the version of statement (iv) with $\ZZ$ replaced by $\ZZ_+$ is dealt with via the continuous mapping theorem.

\begin{claim}
\label{CL:moment}
If (iii$_+$), then for every $t \in \ZZ_+$,
\[
  \E [ \norm{\Theta_t}^\alpha ] = \lim_{r \downto 0} \lim_{u \to \infty} \Pr( \norm{X_{-t}} > ru \mid \norm{X_0} > u ).
\]
\end{claim}

\begin{proof}[\it Proof of Claim~\ref{CL:moment}]
Fix $t \in \ZZ_+$ and $r > 0$. Put $V(u) = \Pr(\norm{X_0} > u)$. By stationarity and (iii+), since $Y$ and $\Theta_t$ are independent and the distribution of $Y$ is continuous,
\begin{multline*}
  \Pr( \norm{X_{-t}} > ru \mid \norm{X_0} > u )
  = \frac{V(ru)}{V(u)} \Pr(\norm{X_t} > u \mid \norm{X_0} > ru) \\
  \to r^{-\alpha} \Pr(rY \norm{\Theta_t} > 1) \qquad (u \to \infty).
\end{multline*}
By independence, this is equal to
\begin{multline}
\label{E:moment:aux}
  \E \biggl[ r^{-\alpha} \int_1^\infty \1(ry \norm{\Theta_t} > 1) \, d(-y^{-\alpha}) \biggr] \\
  = \E \biggl[ \int_r^\infty \1(z \norm{\Theta_t} > 1) \, d(-z^{-\alpha}) \biggr]
  = \E [ \min(\norm{\Theta_t}, 1/r)^\alpha ].
\end{multline}
By monotone convergence, the limit as $r \downto 0$ is $\E [ \norm{\Theta_t}^\alpha ]$, as required.
\end{proof}

\begin{claim}
\label{CL:minust}
If (iii$_+$), then for every $t \in \ZZ_+$,
\[
  \law ( X_{-t} / \norm{X_0} \mid \norm{X_0} > u ) \dto \nu_t \qquad (u \to \infty),
\]
where $\nu_t$ is a probability measure on $\Banach$ given for $\nu_t$-integrable $g : \Banach \to \RR$ by
\[
  \int g \, d\nu_t = g(0) (1 - \E [ \norm{\Theta_t}^\alpha ] ) + \E [ g(\Theta_0 / \norm{\Theta_t}) \, \norm{\Theta_t}^\alpha ].
\]
(The expectation on the right is to interpreted as zero if $\norm{\Theta_t} = 0$.)
\end{claim}

\begin{proof}[\it Proof of Claim~\ref{CL:minust}]
Let $g : \Banach \to \RR$ be continuous and bounded. Fix $r > 0$ (later on, we will take the limit as $r \downto 0$). We have
\begin{multline*}
  \E [ g(X_{-t} / \norm{X_0}) \mid \norm{X_0} > u ] \\
  \shoveleft{= g(0) \, \Pr( \norm{X_{-t}} \le ru \mid \norm{X_0} > u )} \\
  + \E [ \{ g(X_{-t} / \norm{X_0}) - g(0) \} \, \1( \norm{X_{-t}} \le ru ) \mid \norm{X_0} > u] \\
  + \E [ g(X_{-t} / \norm{X_0}) \, \1( \norm{X_{-t}} > ru ) \mid \norm{X_0} > u].
\end{multline*}
The first term on the right-hand side has been treated in Claim~\ref{CL:moment}. Secondly, if $\norm{X_0} > u$ and $\norm{X_{-t}} \le ru$, then $\norm{X_{-t} / \norm{X_0}} < r$. Since $g$ is continuous,
\begin{multline*}
  \lim_{r \downto 0} \limsup_{u \to \infty} \bigl| \E [ \{ g(X_{-t} / \norm{X_0}) - g(0) \} \, \1( \norm{X_{-t}} \le ru ) \mid \norm{X_0} > u] \bigr| \\
  \leq \lim_{r \downto 0} \sup_{x \in B_{0,r}} | g(x) - g(0) | = 0.
\end{multline*}
Thirdly, writing $V(u) = \Pr( \norm{X_0} > u )$, we have, by stationarity,
\begin{align*}
  \lefteqn{
  \E [ g(X_{-t} / \norm{X_0}) \, \1( \norm{X_{-t}} > ru ) \mid \norm{X_0} > u ] 
  } \\
  &= \frac{V(ru)}{V(u)} \E [ g(X_0 / \norm{X_t}) \, \1 (\norm{X_t} > u) \mid \norm{X_0} > ru ] \\
  &= \frac{V(ru)}{V(u)} \E \biggl[ g \biggl( \frac{X_0 / \norm{X_0}}{\norm{X_t} / \norm{X_0}} \biggr) \, 
    \1 \biggl( r \frac{\norm{X_0}}{ru} \frac{\norm{X_t}}{\norm{X_0}} > 1 \biggr) \, \bigg| \, \norm{X_0} > ru \biggr].
\end{align*}
By (iii$_+$), continuity of the law of $Y$ and independence of $Y$ and $\Theta_t$, this converges as $u \to \infty$ to
\[
  r^{-\alpha} \E [ g(\Theta_0 / \norm{\Theta_t}) \, \1 ( r Y \norm{\Theta_t} > 1 ) ]
\]
By the same argument as in \eqref{E:moment:aux}, this is equal to
\[
  \E [ g(\Theta_0 / \norm{\Theta_t}) \, \min(\norm{\Theta_t}, 1/r)^\alpha ],
\]
which tends to $\E [ g(\Theta_0 / \norm{\Theta_t}) \, \norm{\Theta_t}^\alpha ]$ as $r \downto 0$ (dominated convergence). Claim~\ref{CL:minust} is thereby established.
\end{proof}

Fix nonnegative integer $s$ and $t$. If (iii$_+$), then in view of Claim~\ref{CL:minust}, the converse half of Prohorov's theorem \citep[Theorem~6.2]{BIL68} and Tychonoff's theorem, there exists $u_0 > 0$ such that the collection of probability measures
\begin{equation}
\label{E:laws:u}
  \law \bigl( X_{-s} / \norm{X_0}, \ldots, X_t / \norm{X_0} \mid \norm{X_0} > u \bigr), \qquad u > u_0,
\end{equation}
is \defn{tight}, that is, for every $\eps > 0$ there exists a compact subset $K_\eps$ of $\Banach^{t+s+1}$ so that the probability mass of $K_\eps$ under each of the laws above is at least $1 - \eps$. By the direct half of Prohorov's theorem \citep[Theorem~6.1]{BIL68}, the collection of probability measures above is \defn{relatively compact}: for every sequence $u_n \to \infty$ there exists a subsequence $u_{n_m} \to \infty$ for which the laws have a limit in distribution. To prove convergence in distribution of \eqref{E:laws:u} as $u \to \infty$, it is then sufficient to show uniqueness of the possible sequential limits. As probability distributions are determined by their integrals of bounded, Lipschitz continuous functions \citep[proof of Theorem~1.3]{BIL68}, it is sufficient to show the following claim.

\begin{claim}
\label{CL:uniqueness}
If (iii$_+$), then for every nonnegative integer $s$ and $t$ and for every bounded, Lipschitz continuous function $f : \Banach^{t+s+1} \to \RR$, the following limit exists:
\begin{equation}
\label{E:limit2exist}
  \lim_{u \to \infty} \E [ f( X_{-s}/ \norm{X_0}, \ldots, X_t / \norm{X_0}) \mid \norm{X_0} > u].
\end{equation}
\end{claim}

\begin{proof}[\it Proof of Claim~\ref{CL:uniqueness}]
We proceed by induction on $s$. If $s = 0$, there is nothing to prove, for (iii$_+$) already states convergence in distribution as $u \to \infty$.

Let $s \ge 1$ be integer and assume the stated convergence holds for $s-1$, all nonnegative integer $t$, and all bounded, Lipschitz continuous functions from $\Banach^{t+s}$ into $\RR$. Let $f : \Banach^{t+s+1} \to \RR$ be bounded and Lipschitz continuous. Define $f_0 : \Banach^{t+s+1} \to \RR$ by
\[
  f_0(\theta_{-s}, \ldots, \theta_t) = f(\theta_{-s}, \ldots, \theta_t) - f(0, \theta_{-s+1}, \ldots, \theta_t).
\]
We have
\begin{multline*}
  \E [ f(X_{-s} / \norm{X_0}, \ldots, X_t / \norm{X_0}) \mid \norm{X_0} > u ] \\
  = \E [ f_0(X_{-s} / \norm{X_0}, \ldots, X_t / \norm{X_0}) \mid \norm{X_0} > u] \\
  + \E [ f(0, X_{-s+1}/ \norm{X_0}, \ldots, X_t / \norm{X_0}) \mid \norm{X_0} > u].
\end{multline*}
By the induction hypothesis, the limit
\[
  \lim_{u \to \infty} \E [ f(0, X_{-s+1}/ \norm{X_0}, \ldots, X_t / \norm{X_0}) \mid \norm{X_0} > u]
\]
exists. It remains to show the existence of the limit in \eqref{E:limit2exist} with $f$ replaced by $f_0$. Besides being bounded and Lipschitz continuous, the function $f_0$ has the additional property that $f_0(0, \theta_{-s+1}, \ldots, \theta_t) = 0$.

Let $r > 0$ (later on, we will take the limit as $r \downto 0$). Write
\begin{multline*}
  \E [ f_0(X_{-s} / \norm{X_0}, \ldots, X_t / \norm{X_0}) \mid \norm{X_0} > u] \\
  = \E [ f_0(X_{-s} / \norm{X_0}, \ldots, X_t / \norm{X_0}) \, \1(\norm{X_{-s}} \le ru) \mid \norm{X_0} > u] \\
  + \E [ f_0(X_{-s} / \norm{X_0}, \ldots, X_t / \norm{X_0}) \, \1(\norm{X_{-s}} > ru) \mid \norm{X_0} > u].
\end{multline*}
Let $L$ be the Lipschitz constant associated to $f_0$. Consider the first expectation on the right-hand side of the previous display. Since the integrand vanishes if $X_{-s} = 0$, this expectation is bounded by $L \, \norm{X_{-s}} / \norm{X_0} < L \, r$, which converges to $0$ as $r \downto 0$. So it remains to show existence of the limit of the last expectation in the previous display as $u \to \infty$ and then as $r \downto 0$.

Let $V(u) = \Pr(\norm{X_0} > u)$. By stationarity,
\begin{multline*}
  \E [ f_0(X_{-s} / \norm{X_0}, \ldots, X_t / \norm{X_0}) \, \1(\norm{X_{-s}} > ru) \mid \norm{X_0} > u] \\
  \shoveleft{= \frac{V(ru)}{V(u)} \E [ f_0(X_0 / \norm{X_s}, \ldots, X_{t+s} / \norm{X_s}) \, \1(\norm{X_s} > u) \mid \norm{X_0} > ru ]} \\
  \shoveleft{= \frac{V(ru)}{V(u)} \E \biggl[ f_0 \biggl( \frac{X_0 / \norm{X_0}}{\norm{X_s} / \norm{X_0}}, \ldots, \frac{X_{t+s} / \norm{X_0}}{\norm{X_s} / \norm{X_0}} \biggr)} \\
  \1 \biggl( r \frac{\norm{X_0}}{ru} \frac{\norm{X_s}}{\norm{X_0}} > 1\biggr) \, \bigg| \, \norm{X_0} > ru \biggr].
\end{multline*}
In view of (iii$_+$), this converges as $u \to \infty$ to
\[
  r^{-\alpha} \E \biggl[ f_0 \biggl( \frac{\Theta_0}{\norm{\Theta_s}}, \ldots, \frac{\Theta_{t+s}}{\norm{\Theta_s}} \biggr) \, \1 (r Y \norm{\Theta_s} > 1) \biggr].
\]
By a similar argument as the one at the end of the proof of Claim~\ref{CL:moment}, this is equal to the left-hand side in
\begin{multline}
\label{E:TimeChange}
  \E \biggl[ f_0 \biggl( \frac{\Theta_0}{\norm{\Theta_s}}, \ldots, \frac{\Theta_{t+s}}{\norm{\Theta_s}} \biggr) \, \min ( \norm{\Theta_s}, 1/r)^\alpha \biggr] \\
  \to \E \biggl[ f_0 \biggl( \frac{\Theta_0}{\norm{\Theta_s}}, \ldots, \frac{\Theta_{t+s}}{\norm{\Theta_s}} \biggr) \, \norm{\Theta_s}^\alpha \biggr] \qquad (r \downto 0),
\end{multline}
the convergence being justified by dominated convergence. This finishes the proof of Claim~\ref{CL:uniqueness}.
\end{proof}

By Claim~\ref{CL:uniqueness} and the tightness argument preceding it, if (iii$_+$) then the limit in distribution
\[
  \law \bigl( X_{-s} / \norm{X_0}, \ldots, X_t / \norm{X_0} \mid \norm{X_0} > u \bigr) \qquad (u \to \infty),
\]
exists for all nonnegative integer $s$ and $t$. By the Daniell--Kolmogorov extension theorem \citep[Chapter~4, Theorem~53]{POL02}, these limits in distributions are the `finite-dimensional' distributions of a random element $(\Theta_t)_{t \in \ZZ}$ in the product space $\Banach^\ZZ$. Statement (iii) concerning weak convergence in $\Banach^\ZZ$ then follows from the convergence in the previous display for all $s$ and $t$ together with Theorem~1.4.8 in \citet{VDV96}.

It remains to show equation~\eqref{E:time}. By \eqref{E:TimeChange}, equation~\eqref{E:time} holds if $f$ is bounded and Lipschitz continuous and vanishes on the set $\{ (\theta_{-s}, \ldots, \theta_t) \in \Banach^{t+s+1} : \theta_{-s} = 0 \}$. For a general bounded and Lipschitz continuous function $g : \Banach^{t+s+1} \to \RR$, write
\begin{align*}
  g(\Theta_{-s}, \ldots, \Theta_t) 
  &= g(\Theta_{-s}, \ldots, \Theta_t) - g(0, \Theta_{-s+1}, \ldots, \Theta_t) \\
  &\quad \mbox{} + g(0, \Theta_{-s+1}, \ldots, \Theta_t) - g(0, 0, \Theta_{-s+2}, \ldots, \Theta_t) \\
  &\quad \mbox{} + \ldots \\
  &\quad \mbox{} + g(0, \ldots, 0, \Theta_{-1}, \ldots, \Theta_t) - g(0, \ldots, 0, \Theta_0, \ldots, \Theta_t) \\
  &\quad \mbox{} + g(0, \ldots, 0, \Theta_0, \ldots, \Theta_t).
\end{align*}
Take expectations on both sides and apply \eqref{E:TimeChange} to the first $s$ lines of the right-hand side of the previous display at $s$ replaced by $s, s-1, \ldots, 1$, respectively, to obtain
{\small
\begin{align*}
  \lefteqn{\E [ g(\Theta_{-s}, \ldots, \Theta_t) ]} \\
  &= \E \biggl[ \biggl\{ g \biggl( \frac{\Theta_0}{\norm{\Theta_s}}, \ldots, \frac{\Theta_{t+s}}{\norm{\Theta_s}} \biggr) 
    - g \biggl( 0, \frac{\Theta_1}{\norm{\Theta_s}}, \ldots, \frac{\Theta_{t+s}}{\norm{\Theta_s}} \biggr) \biggr\} \, \norm{\Theta_s}^\alpha \biggr] \\
  &\quad \mbox{} + \E \biggl[ \biggl\{ g \biggl( 0, \frac{\Theta_0}{\norm{\Theta_{s-1}}}, \ldots, \frac{\Theta_{t+s-1}}{\norm{\Theta_{s-1}}} \biggr) 
    - g \biggl( 0, 0, \frac{\Theta_1}{\norm{\Theta_{s-1}}}, \ldots, \frac{\Theta_{t+s-1}}{\norm{\Theta_{s-1}}} \biggr) \biggr\} \, \norm{\Theta_{s-1}}^\alpha \biggr] \\
  &\quad \mbox{} + \ldots \\
  &\quad \mbox{} + \E \biggl[ \biggl\{ g \biggl( 0, \ldots, 0, \frac{\Theta_0}{\norm{\Theta_{1}}}, \ldots, \frac{\Theta_{t+1}}{\norm{\Theta_{1}}} \biggr) 
    - g \biggl( 0, \ldots, 0, \frac{\Theta_1}{\norm{\Theta_{1}}}, \ldots, \frac{\Theta_{t+1}}{\norm{\Theta_{1}}} \biggr) \biggr\} \, \norm{\Theta_{1}}^\alpha \biggr] \\
  &\quad \mbox{} + \E [ g(0, \ldots, 0, \Theta_0, \ldots, \Theta_t) ].
\end{align*}}
The equality in the preceding display being true for all bounded and Lipschitz continuous functions $g : \Banach^{t+s+1} \to \RR$, it must hold whenever $g$ is the indicator function of a closed set \citep[proof of Theorem~1.3]{BIL68} and then by a standard argument for all measurable functions $\Banach^{t+s+1} \to \RR$ that are integrable with respect to the law of $(\Theta_{-s}, \ldots, \Theta_t)$. For such functions that vanish whenever their first argument is equal to zero, the formula in the preceding display simplifies to~\eqref{E:time}.

This concludes the proof of Theorem~\ref{T:time}.
\end{proof}

\section{Use of the spectral process}
\label{S:uses}

Various aspects of the extremal dynamics of a regularly varying stationary time series can be expressed in terms of its spectral process. Here we give examples involving joint survival functions, tail dependence coefficients, the extremogram, extremal indices, and point processes. In \cite{BKS10}, the spectral process is also used to prove a functional limit theorem for partial sums of real-valued stationary time series with infinite variance. 

Let $(X_t)_{t \in \ZZ}$ be a stationary time series in a separable Banach space $\Banach$ with dual space $\dual$, that is, the linear space of bounded linear functionals $b^* : \Banach \to \RR$ equipped with the norm $\norm{b^*} = \sup \{ |b^* x| : x \in \sphere \}$. Assume $(X_t)_{t \in \ZZ}$ is regularly varying with index $\alpha$ and spectral process $(\Theta_t)_{t \in \ZZ}$. Let $(Y_t)_{t \in \ZZ}$ be the tail process of $(X_t)_{t \in \ZZ}$, see~\eqref{E:TailProcess}.

\begin{example}[Joint survival functions]
\label{ex:JointSurvFunc}
Fix an integer $t \ge 0$, an index set $I \subset \{0, \ldots, t\}$ such that $0 \in I$, and nonzero bounded linear functionals $b_i^* \in \dual$ ($i \in I$). We are interested in the joint tail behavior of the random variables $b_i^* X_i$ for $i \in I$. By conditioning on the event $\norm{X_0} > u / \norm{b_0^*}$, we find
\begin{align}
\label{E:JointSurvFunc:dual}
  \lim_{u \to \infty} \frac{\Pr( \forall i \in I : b_i^* X_i > u )}{\Pr( \norm{X_0} > u )}
  &= \norm{b_0^*}^\alpha \, \Pr ( \forall i \in I : b_i^* Y_i > \norm{b_0^*} ) \\
  &= \int_{1/\norm{b_0^*}}^\infty \Pr( \forall i \in I : r \, b_i^* \Theta_i > 1 ) \, d(-r^{-\alpha}) \nonumber \\
  &= \E [ \min \{ (b_i^* \Theta_i)_+^\alpha : i \in I \} ]. \nonumber
\end{align}
If we had started to calculate the limit on the left-hand side of \eqref{E:JointSurvFunc:dual} by conditioning on the event $\norm{X_t} > u / \norm{b_t^*}$ instead, we would have found
\begin{equation}
\label{E:JointSurvFunc:dual:bis}
  \lim_{u \to \infty} \frac{\Pr( \forall i \in I : b_i^* X_i > u )}{\Pr( \norm{X_0} > u )}
  = \E [ \min \{ (b_i^* \Theta_{i-t})_+^\alpha : i \in I \} ].
\end{equation}
The equality of the two expectations in equations \eqref{E:JointSurvFunc:dual} and \eqref{E:JointSurvFunc:dual:bis} is a special instance of the time-change formula \eqref{E:time}.

In the same way, one can prove that for $t$ and $I$ as above and for positive constants $b_i$ ($i \in I$),
\begin{align}
\label{E:JointSurvFunc:norm}
  \lim_{u \to \infty} \frac{\Pr( \forall i \in I : b_i \norm{X_i} > u )}{\Pr( \norm{X_0} > u )}
  &= \E [ \min \{ b_i^\alpha \norm{\Theta_i}^\alpha : i \in I \} ] \\
  &= \E [ \min \{ b_i^\alpha \norm{\Theta_{i-t}}^\alpha : i \in I \} ]. \nonumber
\end{align}
\end{example}

\begin{example}[Tail dependence coefficients]
\label{ex:TailDependence}
Let $b^* \in \Banach^*$. If $\Pr(b^* \Theta_0 > 0) > 0$, then the \defn{coefficient of upper tail dependence} between $b^* X_0$ and $b^* X_h$ is given by
\begin{align}
\label{E:TailDependence:dual}
  \lim_{u \to \infty} \Pr( b^* X_h > u \mid b^* X_0 > u )
  &= \frac{\E [ \min \{ (b^* \Theta_0)_+^\alpha, \, (b^* \Theta_{h})_+^\alpha \} ]}{\E [ (b^* \Theta_0)_+^\alpha ] } \\
  &= \frac{\E [ \min \{ (b^* \Theta_0)_+^\alpha, \, (b^* \Theta_{-h})_+^\alpha \} ]}{\E [ (b^* \Theta_0)_+^\alpha ] }. \nonumber
\end{align}
This is an immediate consequence of equations \eqref{E:JointSurvFunc:dual} and \eqref{E:JointSurvFunc:dual:bis}. Likewise, by \eqref{E:JointSurvFunc:norm}, the coefficient of tail dependence between $\norm{X_0}$ and $\norm{X_h}$ is given by
\[
  \lim_{u \to \infty} \Pr( \norm{X_h} > u \mid \norm{X_0} > u ) 
  = \E [ \min(\norm{\Theta_{h}}^\alpha, 1)]
  = \E [ \min(\norm{\Theta_{-h}}^\alpha, 1)].
\]
\end{example}

\begin{example}[Extremogram]
\label{ex:Extremogram}
In \cite{DM09}, the \defn{extremogram} was introduced as an extreme-value analogue of the correllogram through
\[
  \rho_{A, B}(h) = \lim_{n \to \infty} n \, \Pr (X_0/a_n \in A, \, X_h/a_n \in B),
\]
for integer $h$ and for regions $A, B$ at least one of which stays away from the origin and where $a_n$ is a positive sequence satisfying $n \Pr( \norm{X_0} > a_n ) \to 1$ as $n \to \infty$. For instance, if $A = \{ x \in \Banach : a^* x > 1 \}$ and $B = \{ x \in \Banach : b^* x > 1 \}$ for some $a^*, b^* \in \dual$, then by \eqref{E:JointSurvFunc:dual},
\[
  \rho_{A, B}(h) 
  = \lim_{n \to \infty} n \, \Pr ( a^* X_0 > a_n, \, b^* X_h > a_n )
  = \E [ \min \{ (a^* \Theta_0)_+^\alpha, \, (b^* \Theta_h)_+^\alpha \} ].
\]
More generally, if $A$ and $B$ are continuity sets of the distributions of $Y_0$ and $Y_h$ respectively and if $A \subset \{x \in \Banach : \norm{x} > 1 \}$, then
\begin{align*}
  \rho_{A, B}(h) 
  &= \lim_{n \to \infty} \Pr (X_0/a_n \in A, \, X_h/a_n \in B \mid \norm{X_0} > a_n) \\
  &= \Pr(Y_0 \in A, \, Y_h \in B).
\end{align*}
\end{example}

\begin{example}[Extremal indices]
\label{ex:ExtremalIndex}
Let $b^* \in \dual$ be such that $\Pr(b^* \Theta_0 > 0) > 0$. Under assumptions similar to Conditions~4.1 and~4.4 in \cite{BS09}, the \defn{extremal index} \citep{LB83} of the univariate sequence $(b^* X_t)_{t \in \ZZ}$ is given by
\begin{align}
\label{E:ExtremalIndex:dual}
  \theta(b^*)
  &= \lim_{m \to \infty} \lim_{u \to \infty} \Pr \biggl( \max_{t=1, \ldots, m} b^* X_t \le u \, \bigg| \, b^* X_0 > u \biggr) \\
  &= 1 - \frac{\E [ \min \{ (b^* \Theta_0)_+^\alpha, \sup_{t \ge 1} (b^* \Theta_t)_+^\alpha \} ]}{\E [ (b^* \Theta_0)_+^\alpha ]} \nonumber \\
  &= \frac{\E [ \sup_{t \ge 0} (b^* \Theta_t)_+^\alpha - \sup_{t \ge 1} (b^* \Theta_t)_+^\alpha ]}{\E [ (b^* \Theta_0)_+^\alpha ]}. \nonumber
\end{align}
In the previous display, the second equality is justified by an argument similar to the one in \eqref{E:JointSurvFunc:dual}, whereas the last equality is a consequence of the identity $x - \min(x, y) = \max(x, y) - y$ for $x, y \in \RR$. Similarly, the extremal index of the univariate sequence $(\norm{X_t})_{t \in \ZZ}$ is given by
\begin{align}
\label{E:ExtremalIndex:norm}
  \theta 
  &= \lim_{m \to \infty} \lim_{u \to \infty} \Pr \biggl( \max_{t=1, \ldots, m} \norm{X_t} \le u \, \bigg| \, \norm{X_0} > u \biggr) \\
  &= \Pr \biggl( \sup_{t \ge 1} \norm{Y_t} \le 1 \biggr) 
  = \E \biggl[ \sup_{t \ge 0} \norm{\Theta}^\alpha - \sup_{t \ge 1} \norm{\Theta}^\alpha \biggr], \nonumber
\end{align}
see \citet[Remark~4.7]{BS09}.
\end{example}

\begin{example}[Point processes]
\label{ex:PointProcesses}
Consider the sequence of point processes
\[
  N_n = \sum_{i=1}^n \delta_{X_i/a_n},
\]
with $a_n$ a positive sequence such that $n \Pr( \norm{X_0} > a_n ) \to 1$ as $n \to \infty$. We view $N_n$ as a random point measure in $\Mzero$. For Euclidean state spaces, the limit distribution of $N_n$ has been studied in, among others, \cite{DH95}, \cite{DM98}, and \cite{BS09}. Under conditions similar to the ones of Theorem~4.5 in \cite{BS09}, it can be shown that $N_n$ converges weakly to a certain compound Poisson process $N$. The clusters in $N$ are rescaled independent copies of the \defn{cluster process} whose law is equal to
\[
  \law \biggl( \sum_t \delta_{Y_t} \, \bigg| \, \sup_{t \le -1} \norm{Y_t} \le 1 \biggr).
\]
Up to minor changes, the proof is the same as in \cite{BS09} and is omitted for brevity.
\end{example}

\section{Bounded linear operators}
\label{S:linear}

Let $\Banach_1$ and $\Banach_2$ be real, separable Banach spaces. We consider the effect of a bounded linear operator $A : \Banach_1 \to \Banach_2$ on the regular variation properties of a $\Banach_1$-valued random element $X$ (Proposition~\ref{P:linear:rv}) and a stationary time series $(X_t)_{t \in \ZZ}$ (Proposition~\ref{P:linear:series}). The case of bounded linear functionals arises when $\Banach_2 = \RR$. We do not pursue generalizations of Breiman's lemma involving operators that are themselves random; see e.g.\ \citet[Proposition~A.1]{B02b} for the finite-dimensional case.

Let $\sphere_j$ denote the unit sphere in $\Banach_j$ ($j = 1, 2$). The norm in both Banach spaces will be denoted by the symbol $\norm{\,\cdot\,}$.

\begin{proposition}[Linear transformation of random elements]
\label{P:linear:rv}
Let $X$ be a regularly varying random element in $\Banach_1$ with index $\alpha > 0$ and spectral measure $\lambda$ and let $A : \Banach_1 \to \Banach_2$ be a bounded linear operator. We have
\begin{equation}
\label{E:linear:tail}
  \frac{\Pr( \norm{AX} > u )}{\Pr(\norm{X} > u)} \to \int_{\sphere_1} \norm{A\theta}^\alpha \, \lambda(d\theta) \qquad (u \to \infty).
\end{equation}
If $\lambda( \{ \theta \in \sphere_1 : A\theta \neq 0 \} ) > 0$, this limit is positive and $AX$ is regularly varying in $\Banach_2$ with the same index $\alpha$ and with spectral measure $\lambda_A$ given by
\begin{equation}
\label{E:lambda:A}
  \int_{\sphere_2} g(\theta) \, \lambda_A(d\theta)
  = \frac{1}{\int_{\sphere_1} \norm{A\theta}^\alpha \, \lambda(d\theta)} \int_{\sphere_1} 
  g \biggl( \frac{A\theta}{\norm{A\theta}} \biggr) \, \norm{A\theta}^\alpha \, \lambda(d\theta).
\end{equation}
for $\lambda_A$-integrable $g : \sphere_2 \to \RR$. (For $\theta$ such that $\norm{A\theta} = 0$, the integrand on the right is to be interpreted as zero.)
\end{proposition}

\begin{proof}
Let $\mu$ be the limit measure of $X$ when $V(u) = \Pr(\norm{X} > u)$; see \eqref{E:lambda2mu:f}. As $\mu$ is homogeneous, the set $\{ x \in \Banach_1 : \norm{Ax} = 1 \}$ is a $\mu$-null set. As a consequence, as $u \to \infty$,
\begin{multline*}
  \frac{\Pr(\norm{AX} > u)}{\Pr(\norm{X} > u)} 
  \to \mu( \{ x \in \Banach_1 : \norm{Ax} > 1 \} ) \\
  = \int_{\sphere_1} \int_0^\infty \1(r \norm{A\theta} > 1) \, d(-r^{-\alpha}) \, \lambda(d\theta)
  = \int_{\sphere_1} \norm{A\theta}^\alpha \, \lambda(d\theta).
\end{multline*}
If $\lambda( \{ \theta : A\theta \neq 0 \} ) > 0$, the integral on the right is positive. Let $f \in \Czero(\Banach_2)$ and let $z > 0$ be such that $f$ vanishes on the ball $B_{0,z}$ in $\Banach_2$. Since $A$ is bounded, the function $\Banach_1 \to \RR : x \mapsto f(Ax)$ belongs to $\Czero(\Banach_1)$. By regular variation of $X$ in $\Banach_1$,
\[
  \frac{1}{\Pr(\norm{X} > u)} \E [ f(AX/u) ] \to \int_{\Banach_1} f(Ax) \, \mu(dx) \qquad (u \to \infty),
\]
and $AX$ is regularly varying in $\Banach_2$ with limit measure $\mu_A = \mu \circ A^{-1}$. For measurable $g : \sphere_2 \to [0, \infty)$, by \eqref{E:lambda2mu:f},
\begin{align*}
  \lefteqn{
  \int_{\Banach_2} g(y/\norm{y}) \, \1(\norm{y} > 1) \, \mu_A(dy) 
  } \\
  &= \int_{\Banach_2} g(Ax / \norm{Ax}) \1(\norm{Ax} > 1) \, \mu(dx) \\
  &= \int_{\sphere_1} \int_0^\infty g(A\theta / \norm{A\theta}) \1(r \norm{A\theta} > 1) \, d(-r^{-\alpha}) \, \lambda(d\theta) \\ 
  &= \int_{\sphere_1} g(A\theta / \norm{A\theta}) \norm{A\theta}^\alpha \, \lambda(d\theta),
\end{align*}
yielding~\eqref{E:lambda:A}.
\end{proof}

The expression for $\lambda_A$ in \eqref{E:lambda:A} has the following probabilistic meaning: if $\Theta$ is a random element in $\sphere_1$ with distribution $\lambda$ and if $U$ is a $\text{Uniform}(0,1)$ random variable independent of $\Theta$, then
\begin{equation}
\label{E:lambda:A:prob}
  \lambda_A = \law \biggl( \frac{A \Theta}{\norm{A \Theta}} \, \bigg| \, U \le \frac{\norm{A \Theta}^\alpha}{\norm{A}^\alpha} \biggr).
\end{equation}
To show \eqref{E:lambda:A:prob}, it suffices to check that $\E [ g( A \Theta / \norm{A \Theta} ) \mid U \le \norm{A \Theta}^\alpha / \norm{A}^\alpha ]$ is equal to the right-hand side of \eqref{E:lambda:A}. Equation~\eqref{E:lambda:A:prob} justifies the following rejection algorithm to generate a random draw $\Theta_A$ from $\lambda_A$:
\begin{equation}
\label{Algorithm:linear}
\text{
\begin{minipage}[t]{0.87\textwidth}
\begin{list}{}{\setlength{\leftmargin}{3ex}}
\item[1.] Draw $\Theta \sim \lambda$ and $U \sim \text{Uniform}(0,1)$ independently.
\item[2.] If $U \le \norm{A \Theta}^\alpha / \norm{A}^\alpha$, then return $\Theta_A = A \Theta / \norm{A \Theta}$ and stop.
\item[3.] Otherwise, go back to step 1.
\end{list}
\end{minipage}
}
\end{equation}
In \eqref{E:lambda:A:prob} and \eqref{Algorithm:linear}, the operator norm $\norm{A}$ may be replaced by the possibly smaller essential supremum of $\norm{A \Theta}$, that is, by the left-hand side in
\begin{equation}
\label{E:EssSup}
  \inf \{ y > 0 : \Pr(\norm{A \Theta} > y) = 0 \} \le \norm{A}.
\end{equation}

\begin{example}[Isometries]
\label{ex:Isometry}
If the linear operator $A$ satisfies $\norm{Ax} = \norm{A} \norm{x}$ for all $x \in \Banach_1$, that is, if the linear operator $A / \norm{A}$ is an isometry, then the limit in \eqref{E:linear:tail} is just $\norm{A}^\alpha$, and \eqref{E:lambda:A} simplifies to
\[
\label{E:lambda:A:Isometry}
  \int_{\sphere_2} g(\theta) \, \lambda_A(d\theta) = \int_{\sphere_1} g(A\theta / \norm{A}) \, \lambda(d\theta),
\]
that is, $\lambda_A$ is the distribution of $A \Theta / \norm{A}$, with $\Theta$ a random element in $\sphere_1$ having law $\lambda$.
\end{example}

\begin{proposition}[Linear transformations of time series]
\label{P:linear:series}
Let $(X_t)_{t \in \ZZ}$ be a stationary time series in $\Banach_1$ and let $A : \Banach_1 \to \Banach_2$ be a bounded linear operator. If $(X_t)_{t \in \ZZ}$ is regularly varying with index $\alpha > 0$ and spectral process $(\Theta_t)_{t \in \ZZ}$ and if $A \Theta_0$ is not degenerate at zero, then $(AX_t)_{t \in \ZZ}$ is regularly varying with index $\alpha > 0$ too and the law of its spectral process $(\Theta_t^A)_{t \in \ZZ}$ is given by
\begin{multline}
\label{E:ThetaA}
  \E [ f(\Theta_{-s}^A, \ldots, \Theta_t^A) ] \\
  = \frac{1}{\E [ \norm{A\Theta_0}^\alpha ]} \,
  \E \biggl[ f \biggl( \frac{A \Theta_{-s}}{\norm{A \Theta_0}}, \ldots, \frac{A \Theta_t}{\norm{A \Theta_0}} \biggr) \, \norm{A \Theta_0}^\alpha \biggr]
\end{multline}
for all nonnegative integer $s$ and $t$ and all measurable functions $f : \Banach_2^{t+s+1} \to \RR$ for which at least one of the two expectations is defined. (If $\norm{A \Theta_0} = 0$, the integrand on the right is to be interpreted as zero.)
\end{proposition}

\begin{proof}
By Proposition~\ref{P:linear:rv},
\begin{equation}
\label{E:tailA}
  \frac{\Pr( \norm{A X_0} > u )}{\Pr( \norm{X_0} > u )}
  \to \E [ \norm{A \Theta_0}^\alpha ] \qquad (u \to \infty),
\end{equation}
a limit which is strictly positive by assumption. For positive integer $k$, let $\mu_k$ be the limit measure in~\eqref{E:muk}. Then by the previous display and by regular variation of $(X_t)_{t \in \ZZ}$,
\begin{multline*}
  \frac{1}{\Pr( \norm{A X_0} > u )} \, \Pr[(A X_1 / u, \ldots, A X_k / u) \in \, \cdot \,] \\
  \to \frac{1}{\E [ \norm{A \Theta_0}^\alpha ]} \, \mu_k \circ A^{-1} =: \mu_k^A \qquad (u \to \infty) \quad \text{in $\Mzero(\Banach_2^k)$}.
\end{multline*}
It follows that the series $(A X_t)_{t \in \ZZ}$ is regularly varying with index $\alpha > 0$. Let $(\Theta_t^A)_{t \in \ZZ}$ be its spectral process.

Let $f : \Banach_2^k \to \RR$ be bounded and continuous. Since $\norm{AX_0} > u$ implies $\norm{X_0} > u/\norm{A}$, we have
\begin{multline*}
  \E \biggl[ f \biggl( \frac{AX_{-s}}{\norm{AX_0}}, \ldots, \frac{AX_t}{\norm{AX_0}} \biggr) \, \bigg| \, \norm{AX_0} > u \biggr] \\
  \shoveleft = \frac{\Pr(\norm{X_0} > u/\norm{A})}{\Pr( \norm{AX_0} > u)} \\
  \E \biggl[ f \biggl( \frac{AX_{-s}}{\norm{AX_0}}, \ldots, \frac{AX_t}{\norm{AX_0}} \biggr) \, \1(\norm{AX_0} > u) \, \bigg| \, \norm{X_0} > u/\norm{A} \biggr].
\end{multline*}
By \eqref{E:tailA}, regular variation of the function $u \mapsto \Pr(\norm{X_0} > u)$, and Theorem~\ref{T:spectral}(iii), the expression in the preceding display converges as $u \to \infty$ to
\[
  \frac{\norm{A}^\alpha}{\E[ \norm{A\Theta_0}^\alpha ]}
  \E \biggl[ f \biggl( \frac{A\Theta_{-s}}{\norm{A\Theta_0}}, \ldots, \frac{A\Theta_t}{\norm{A\Theta_0}} \biggr) \, \1(Y \norm{A\Theta_0} > \norm{A}) \biggr].
\]
Since $Y$ is a $\Pareto(\alpha)$ random variable independent of the spectral process, we may replace the indicator variable on the right-hand side by the conditional probability $\Pr(Y \norm{A \Theta_0} > \norm{A} \mid \Theta_0) = \norm{A \Theta_0}^\alpha / \norm{A}^\alpha$ (observe that $\norm{A \Theta_0} \le \norm{A}$). We arrive at the right-hand side of~\eqref{E:ThetaA}. 

We have now shown the identity \eqref{E:ThetaA} for functions $f : \Banach_2^k \to \RR$ that are bounded and continuous. As a consequence, the two probability distributions defined by both sides of~\eqref{E:ThetaA} are equal. It follows that~\eqref{E:ThetaA} is true for all functions $f$ for which at least one of the two expectations is defined.
\end{proof}

As in \eqref{E:lambda:A:prob}, we find that the spectral process $(\Theta_t^A)_{t \in \ZZ}$ of $(AX_t)_{t \in \ZZ}$ is given by
\begin{equation}
\label{E:ThetaA:prob}
  \law \bigl( (\Theta_t^A)_{t \in \ZZ} \bigr)
  = \law \Biggl( \biggl(\frac{A \Theta_t}{\norm{A \Theta_0}} \biggr)_{t \in \ZZ} \, \Bigg| \,  U \le \frac{\norm{A \Theta_0}^\alpha}{\norm{A}^\alpha} \Biggr),
\end{equation}
with $U$ a $\text{Uniform}(0, 1)$ random variable independent of $(\Theta_t)_{t \in \ZZ}$. The rejection algorithm in \eqref{Algorithm:linear} can be adapted in an obvious way to produce random draws from the distribution of $(\Theta_t^A)_{t \in \ZZ}$.

\section{Infinite random sums}
\label{S:series}

Let $\Banach_1$ and $\Banach_2$ be real, separable Banach spaces. We are interested in the tail behavior of the $\Banach_2$-valued infinite random sum
\begin{equation}
\label{E:series}
  X = \som_n \, T_n Z_n
\end{equation}
where $(Z_n)_{n \in \ZZ}$ is a sequence of independent and identically distributed random elements in $\Banach_1$ and the $T_n : \Banach_1 \to \Banach_2$ are bounded linear operators. In Section~\ref{S:LinProc}, we will extend this study to linear processes. In Euclidean space, the case of random linear operators (i.e.\ random matrices) has been studied in \citet{HS08}.

We first discuss convergence of the random sum in \eqref{E:series}. Assume that the function $V$ defined by $V(x) = \Pr(\norm{Z_n} > x)$ for $x > 0$ is regularly varying with index $- \alpha$ for some $\alpha > 0$ and that there exists $\delta$ with $0 < \delta < \min(\alpha, 1)$ such that
\begin{equation}
\label{E:Resnick}
  \som_n \norm{T_n}^\delta < \infty.
\end{equation}
As $\E [ \norm{Z_n}^\delta ] = \int_0^\infty V(x^{1/\delta}) \, dx < \infty$, we have
\begin{equation}
\label{E:SeriesConvergence}
  \E [ (\som_n \, \norm{T_n Z_n})^\delta ] \le \som_n \, \norm{T_n}^\delta \, \E [ \norm{Z_n}^\delta ] < \infty,
\end{equation}
so that the series $X$ converges almost surely. Moreover, by \citet[Lemma~4.24]{RES87}, the tail of $\norm{X}$ is of the same order as the one of $\norm{Z_n}$:
\begin{equation}
\label{E:TailBound}
  \frac{\Pr(\norm{X} > x)}{V(x)} \le \frac{\Pr( \som_n \, \norm{T_n} \norm{Z_n} > x )}{V(x)} \to \som_n \, \norm{T_n}^\alpha < \infty
\end{equation}
as $x \to \infty$. If the common law of the random elements $Z_n$ is concentrated on a closed linear subspace $\Banach_{1,0}$ of $\Banach_1$, then \eqref{E:Resnick} may be replaced by the weaker condition
\begin{equation}
\label{E:Resnick:0}
  \som_n \, \norm{T_n}_0^\delta < \infty,
\end{equation}
with $\delta$ as before and with $\norm{T}_0 = \sup \{ \norm{T x} : x \in \Banach_{1,0}, \norm{x} = 1 \}$ the operator norm of the restriction of the linear operator $T : \Banach_1 \to \Banach_2$ to $\Banach_{1,0}$; note that $\norm{T}_0 \le \norm{T}$. Likewise, in \eqref{E:SeriesConvergence} and \eqref{E:TailBound}, $\norm{T_n}$ may be replaced by $\norm{T_n}_0$. Further note that if $\alpha > 1$ and if the innovations have expectation zero ($\Banach_1 = \Banach_2 = \RR$), condition~\eqref{E:Resnick} is not necessary for convergence of the series $X$, although in that case the convergence does not need to hold absolutely; see for instance \citet[Lemma~A.3]{MS00}. 


Now assume that the common distribution of the random elements $Z_n$ is regularly varying with index $\alpha$ and spectral measure $\lambda$. By Proposition~\ref{P:linear:rv}, we have
\begin{equation}
\label{E:series:c_n}
  \lim_{x \to \infty} \frac{\Pr( \norm{T_n Z_n} > x )}{V(x)} = \int_{\sphere_1} \norm{T_n \theta}^\alpha \, \lambda(d\theta) =: c_n.
\end{equation}
Moreover, if $c_n > 0$, then $T_n Z_n$ is regularly varying in $\Banach_2$ with index $\alpha$ and with spectral measure $\lambda_n$ given by
\begin{equation}
\label{E:series:lambda_n}
  \int_{\sphere_2} f(\theta) \, \lambda_n(d\theta) = \frac{1}{c_n} \int_{\sphere_1} f(T_n \theta / \norm{T_n \theta}) \, \norm{T_n \theta}^\alpha \, \lambda(d\theta)
\end{equation}
for $\lambda_n$-integrable functions $f : \sphere_2 \to \RR$. If at least one of the $c_n$ is positive, then according to the next proposition, the random series $X$ is regularly varying with index $\alpha$ as well, its spectral measure $\lambda_X$ being a mixture of the spectral measures $\lambda_n$.

\begin{proposition}
\label{P:series}
If the common distribution of the independent random elements $Z_n$ ($n \in \ZZ$) is regularly varying with index $\alpha$ and spectral measure $\lambda$ and if the linear operators $T_n$ satisfy \eqref{E:Resnick}, then, with $c_n$ as in \eqref{E:series:c_n}, we have
\begin{equation}
\label{E:series:tail}
  \lim_{x \to \infty} \frac{\Pr( \norm{\som_n T_n Z_n} > x )}{V(x)} 
  = \lim_{x \to \infty} \frac{\Pr( \som_n \norm{T_n Z_n} > x)}{V(x)} 
  = \som_n c_n < \infty.
\end{equation}
If $\som_n c_n > 0$, then the random series $X = \som_n T_n Z_n$ is regularly varying with index $\alpha$ too, its spectral measure $\lambda_X$ being given by
\begin{equation}
\label{E:series:lambda}
  \lambda_X = \som_n p_n \lambda_n,
\end{equation}
with $\lambda_n$ as in \eqref{E:series:lambda_n} and with
\begin{equation}
\label{E:series:p_n}
  p_n 
  = \frac{c_n}{\sum_k c_k} 
  = \lim_{x \to \infty} \Pr ( \norm{T_n Z_n} > x \mid \norm{\som_k T_k Z_k} > x ).
\end{equation}
\end{proposition}

\begin{proof}
Equation~\eqref{E:series:tail} follows from equation~\eqref{E:series:c_n} and Proposition~\ref{P:sum:inf}, equation~\eqref{E:sum:inf:tail}.

Let $f : \Banach_2 \to [0, \infty)$ be bounded and continuous. By Proposition~\ref{P:sum:inf}, equation~\eqref{E:sum:inf},
\begin{multline}
\label{E:series:f}
  \E [ f(\som_k T_k Z_k / x) \, \1( \norm{\som_n T_n Z_n} > x ) ] \\
  = \som_n \E [ f(\som_k T_k Z_k / x) \mid \norm{T_n Z_n} > x ] \, \Pr( \norm{T_n Z_n} > x ) + o \bigl( V(x) \bigr)
\end{multline}
as $x \to \infty$. Assume $\sum_n c_n > 0$. Let $n \in \ZZ$ be such that $c_n > 0$. Regular variation of $T_n Z_n$ with index $\alpha$ and spectral measure $\lambda_n$ entails
\[
  \lim_{x \to \infty} \E [ f(T_n Z_n / x) \mid \norm{T_n Z_n} > x ] 
  = \int_1^\infty \int_{\sphere_2} f(r\theta) \, \lambda_n(d\theta) \, d(-r^{-\alpha}).
\]
By independence, we have
\[
  \law \bigl( \som_{k \neq n} T_k Z_k / x \, \big| \, \norm{T_n Z_n} > x \bigr)
  = \law \bigl( \som_{k \neq n} T_k Z_k / x \bigr)
  \dto \delta_0 \qquad (x \to \infty).
\]
Combine the two previous displays to see that
\[
  \lim_{x \to \infty} \E [ f(\som_k T_k Z_k / x) \mid \norm{T_n Z_n} > x ] 
  = \int_1^\infty \int_{\sphere_2} f(r\theta) \, \lambda_n(d\theta) \, d(-r^{-\alpha}).
\]
The latter convergence in combination with \eqref{E:series:tail} and \eqref{E:series:f} implies
\begin{multline*}
  \lim_{x \to \infty} \frac{1}{V(x)} \E [ f(\som_k T_k Z_k / x) \, \1 (\norm{\som_n T_n Z_n} > x ) ] \\
  = \sum_n c_n \int_1^\infty \int_{\sphere_2} f(r\theta) \, \lambda_n(d\theta) \, d(-r^{-\alpha}).
\end{multline*}
We obtain
\begin{multline*}
  \lim_{x \to \infty} \frac{1}{V(x)} \E [ f(\som_k T_k Z_k / x) \mid \norm{\som_n T_n Z_n} > x  ] \\
  = \sum_n p_n \int_1^\infty \int_{\sphere_2} f(r\theta) \, \lambda_n(d\theta) \, d(-r^{-\alpha}).
\end{multline*}
Equation~\eqref{E:series:lambda} follows with $p_n = c_n / \sum_k c_k$. Equation~\eqref{E:series:p_n} now follows from Proposition~\ref{P:sum:inf}.
\end{proof}


\begin{example}[Isometries]
\label{ex:series:Isometry}
Assume that the linear operators $T_i$ satisfy $\norm{T_i z} = \norm{T_i} \norm{z}$ for all $i \in \ZZ$ and all $z \in \Banach_1$, that is, the linear operators $T_i / \norm{T_i}$ are isometries, see Example~\ref{ex:Isometry}. The constants $c_n$ and $p_n$ in \eqref{E:series:c_n} and \eqref{E:series:p_n} respectively then simplify to
\begin{align}
\label{E:series:Isometry:constants}
  c_n &= \norm{T_n}^\alpha, & p_n &= \frac{\norm{T_n}^\alpha}{\sum_k \norm{T_k}^\alpha}
\end{align}
for $n \in \ZZ$. The tail of $\norm{X}$ satisfies
\[
  \lim_{u \to \infty} \frac{\Pr( \norm{X} > u )}{\Pr( \norm{Z_0} > u )} = \som_n \, \norm{T_n}^\alpha.
\]
If $T_n$ is nonzero, the spectral measure $\lambda_n$ in \eqref{E:series:lambda_n} is equal to the law of $T_n \Theta^Z / \norm{T_n}$, with $\Theta^Z$ a random element in $\sphere_1$ with distribution $\lambda$. The spectral measure $\lambda_X$ of $X$ is then equal to the law of $T_N \Theta^Z / \norm{T_N}$, with $N$ an integer-valued random variable, independent of $\Theta^Z$ and with distribution given by $\Pr(N = n) = p_n$ for $n \in \ZZ$.
\end{example}

\begin{example}[Linear combinations with random coefficients]
\label{ex:series:LinComb1}
In case $\Banach_1 = \RR$ we can write $\Banach_2 = \Banach$ and the series $X$ is an infinite linear combination of the elements $\psi_i = T_i(1) \in \Banach$ with random coefficients $Z_i$:
\[
  X = \som_i \, Z_i \psi_i.
\]
Since $\norm{T_i z} = |z| \, \norm{\psi_i}$ for all $z \in \RR$ and all $i \in \ZZ$, Example~\ref{ex:series:Isometry} applies with $\norm{T_i} = \norm{\psi_i}$. In particular, the spectral measure of $X$ is equal to the law of $\Theta^Z \psi_N / \norm{\psi_N}$, with $\Theta^Z$ a random variable in $\{-1, +1\}$ and $N$ an integer-valued random variable independent of $\Theta^Z$ and distribution determined by $\Pr(N = n) = p_n = \norm{\psi_n}^\alpha / \sum_k \norm{\psi_k}^\alpha$ for all $n \in \ZZ$.
\end{example}

\begin{example}[Linear combinations of random elements]
\label{ex:series:LinComb2}
In case $\Banach_1 = \Banach_2 = \Banach$ and $T_i = a_i \Id$ with $a_i \in \RR$ and $\Id$ the identity operator on $\Banach$, the series $X$ is an infinite linear combination of the random elements $Z_i$ with coefficients $a_i$:
\[
  X = \som_i \, a_i Z_i.
\]
Since $\norm{T_i z} = |a_i| \, \norm{z}$ for all $z \in \Banach$ and $i \in \ZZ$, Example~\ref{ex:series:Isometry} applies again with $\norm{T_i} = a_i$. The spectral measure of $X$ is equal to the law of $\sign(a_N) \, \Theta^Z$, with $N$ an integer-valued random variable independent of $\Theta^Z$ and distribution equal to $\Pr(N = n) = p_n = |a_n|^\alpha / \sum_k |a_k|^\alpha$ for $n \in \NN$. Note that
\[
  \Pr[\sign(a_N) = \pm 1] = \frac{\sum_n (a_n)_\pm^\alpha}{\sum_n |a_n|^\alpha},
\]
where $(a)_\pm = \max(\pm a, 0)$ for $a \in \RR$.
\end{example}

\section{Linear processes}
\label{S:LinProc}

Consider the same setting as in Section~\ref{S:series}. Rather than a single random series, we now study the linear process
\begin{equation}
\label{E:LinProc}
  X_t = \som_i \, T_i Z_{t-i}, \qquad t \in \ZZ.
\end{equation}
As before, $(Z_n)_{n \in \ZZ}$ is a sequence of independent and identically distributed random elements in $\Banach_1$ and the $T_n : \Banach_1 \to \Banach_2$ are bounded linear operators. If the common distribution of the $Z_n$ is regularly varying with index $\alpha$ and if Resnick's condition in equation~\eqref{E:Resnick} holds, then the random series defining $X_t$ converges absolutely and $(X_t)_{t \in \ZZ}$ is a stationary time series in $\Banach_2$. Banach-space valued linear processes with light-tailed innovations (finite second moments) have been studied extensively in the field of functional data analysis, see for instance the monograph by \cite{BOS00}. Linear processes with regularly varying innovation distribution appear for instance in \cite{DR85}, \cite{DR86}, \cite{EKM97}, \cite{MS00}, and \cite{DM08}.

Recall $c_n = \int_{\sphere_1} \norm{T_n \theta}^\alpha \, \lambda(d\theta)$ in \eqref{E:series:c_n}, with $\lambda$ the spectral measure of the common law of the random elements $Z_t$. If $c_n > 0$, we can define a probability measure $\kappa_n$ on the space $\Banach_2^\ZZ$ of $\Banach_2$-valued sequences endowed with the product topology by
\begin{multline}
\label{E:LinProc:kappa_n}
  \int_{\Banach_2^\ZZ} f (\theta_{-s}, \ldots, \theta_t) \, \kappa_n \bigl(d (\theta_n)_{n \in \ZZ}\bigr) \\
  = \frac{1}{c_n} \int_{\sphere_1} f \biggl( \frac{T_{-s+n} \theta}{\norm{T_n \theta}}, \ldots, \frac{T_{t+n} \theta}{\norm{T_n \theta}} \biggr) \, 
    \norm{T_n \theta}^\alpha \, \lambda(d\theta),
\end{multline}
for nonnegative integer $s, t$ and for bounded and continuous $f : \Banach_1^{t+s+1} \to \RR$. Extending Proposition~\ref{P:series}, we find that if $\sum_n c_n > 0$, the time series $(X_t)_{t \in \ZZ}$ is regularly varying with index $\alpha > 0$, the law of its spectral measure being a mixture of the $\kappa_n$ above.

\begin{proposition}
\label{P:LinProc}
Let the common distribution of the $\Banach_1$-valued, independent random elements $Z_n$ ($n \in \ZZ$) be regularly varying with index $\alpha$ and spectral measure $\lambda$, and let $T_n : \Banach_1 \to \Banach_2$ be bounded linear operators. If \eqref{E:Resnick} holds and if $\sum_n c_n > 0$ with $c_n$ as in \eqref{E:series:c_n}, then $(X_t)_{t \in \ZZ}$ in \eqref{E:LinProc} is a regularly varying stationary time series in $\Banach_2$ with index $\alpha$, its spectral process $(\Theta_t)_{t \in \ZZ}$ having law $\kappa$ equal to
\begin{equation}
\label{E:LinProc:kappa}
  \kappa = \som_n p_n \kappa_n
\end{equation}
with $p_n$ as in equation~\eqref{E:series:p_n} and $\kappa_n$ as in equation~\eqref{E:LinProc:kappa_n}.
\end{proposition}

The proof of Proposition~\ref{P:LinProc} is similar to the one of Proposition~\ref{P:series} and is omitted. As in \eqref{E:series:lambda}, equation~\eqref{E:LinProc:kappa} constitutes a mixture representation of the distribution $\kappa$ of the spectral process $(\Theta_t)_{t \in \ZZ}$. The algorithm in \eqref{Algorithm:linear} can be adapted in the following way to produce a random draw from $\kappa$:
\begin{equation}
\label{algorithm:kappa}
\text{
\begin{minipage}[t]{0.87\textwidth}
\begin{list}{}{\setlength{\leftmargin}{3ex}}
\item[1.] Draw a random integer $N$ from $(p_n)_{n \in \ZZ}$.
\item[2.] Independently from $N$ and from each other, draw $\Theta^Z \sim \lambda$ and $U \sim \text{Uniform}(0, 1)$.
\item[3.] If $U \le \norm{T_N \Theta^Z}^\alpha / \norm{T_N}^\alpha$, then return $\Theta_t = T_{N+t} \Theta^Z / \norm{T_N \Theta^Z}$ for all $t \in \ZZ$ and stop.
\item[4.] Otherwise, go back to step 2.
\end{list}
\end{minipage}
}
\end{equation}

If $\Theta^Z$ denotes a $\sphere_1$-valued random element whose distribution is equal to the spectral measure $\lambda$ of $Z_n$, then \eqref{E:LinProc:kappa} can be written as
\begin{multline}
\label{E:LinProc:Theta}
  \E [ f(\Theta_{-s}, \ldots, \Theta_t) ] \\
  = \frac{1}{\sum_n \E [ \norm{T_n \Theta^Z}^\alpha ]}
  \sum_n \E \biggl[ f \biggl( \frac{T_{-s+n} \Theta^Z}{\norm{T_n \Theta^Z}}, \ldots, \frac{T_{t+n} \Theta^Z}{\norm{T_n \Theta^Z}} \biggr) \, \norm{T_n \Theta^Z}^\alpha \biggr]
\end{multline}
for nonnegative integer $s, t$ and for integrable $f : \Banach_2^{t+s+1} \to \RR$; as usual, the integrand on the right is to be interpreted as $0$ if $\norm{T_n \Theta^Z} = 0$.

\begin{example}[Joint survival functions]
\label{ex:LinProc:JointSurvfunc}
In Example~\ref{ex:JointSurvFunc}, the asymptotics of joint survival functions of certain random vectors associated to the time series $(X_t)_{t \in \ZZ}$ were calculated in terms of the spectral process. By equations \eqref{E:JointSurvFunc:dual}, \eqref{E:JointSurvFunc:norm} and \eqref{E:LinProc:Theta}, we find
\begin{align}
\label{E:LinProc:Theta:special:dual}
  \lim_{u \to \infty} \frac{\Pr( \forall i \in I : b_i^* X_i > u)}{\Pr(\norm{X_0} > u)}
  &= \frac{\sum_n \E [ \min \{ (b_i^* T_{i+n} \Theta^Z)_+^\alpha : i \in I \} ]}{\sum_n \E [ \norm{T_n \Theta^Z}^\alpha ]},\\
\label{E:LinProc:Theta:special:norm}
  \lim_{u \to \infty} \frac{\Pr( \forall i \in I : b_i \norm{X_i} > u )}{\Pr( \norm{X_0} > u )}
  &= \frac{\sum_n \E [ \min \{ b_i^\alpha \norm{T_{i+n} \Theta^Z}^\alpha : i \in I \} ]}{\sum_n \E [ \norm{T_n \Theta^Z}^\alpha ]},
\end{align}
valid for every finite set $I \subset \ZZ$ with $0 \in I$ and for all linear functionals $b_i^* \in \dual_2$ and positive constants $b_i$ ($i \in I$). If $I = \{0, h\}$ for some nonzero integer $h$, we obtain expressions for certain tail dependence coefficients and extremograms (Examples~\ref{ex:TailDependence} and \ref{ex:Extremogram}).
\end{example}

\begin{example}[Extremal indices]
\label{ex:LinProc:ExtremalIndex}
For $b^* \in \dual_2$ such that $\Pr(b^* T_n \Theta_Z > 0) > 0$ for some $n$, the extremal index $\theta(b^*)$ in \eqref{E:ExtremalIndex:dual} of $(b^* X_t)_{t \in \ZZ}$ is given by
\begin{equation}
\label{E:LinProc:ExtremalIndex:dual}
  \theta(b^*) 
  = 1 - \frac{\E [ \min \{ (b^* \Theta_0)_+^\alpha, \, \sup_{t \ge 1} (b^* \Theta_t)_+^\alpha \} ]}{\E [ (b^* \Theta_0)_+^\alpha ]}
  = \frac{\E [ \sup_n (b^* T_n \Theta^Z)_+^\alpha ]}{\sum_n \E [ (b^* T_n \Theta^Z)_+^\alpha ]}.
\end{equation}
The proof of the second equality is as follows. For the denominator, we have by \eqref{E:LinProc:Theta},
\[
  \E [ (b^* \Theta_0)_+^\alpha ]
  = \frac{\sum_n \E [ (b^* T_n \Theta^Z)_+^\alpha ]}{\sum_n \E [ \norm{T_n \Theta^Z}^\alpha ]}.
\]
For the numerator, we have by monotone convergence and \eqref{E:LinProc:Theta},
\begin{align*}
  \lefteqn{
  \E \biggl[ \min \biggl\{ (b^* \Theta_0)_+^\alpha, \, \sup_{t \ge 1} (b^* \Theta_t)_+^\alpha \biggr\} \biggr] 
  } \\
  &= \lim_{m \to \infty} \E \biggl[ \min \biggl\{ (b^* \Theta_0)_+^\alpha, \, \max_{1 \le t \le m} (b^* \Theta_t)_+^\alpha \biggr\} \biggr] \\
  &= \lim_{m \to \infty} \frac{\sum_n \E [ \min \{ (b^* T_n \Theta^Z)_+^\alpha, \max_{1 \le t \le m} (b^* T_{n+t} \Theta^Z)_+^\alpha \} ] }
    {\sum_n \E [ \norm{T_n \Theta^Z}^\alpha ]} \\
  &= \frac{\sum_n \E [ \min \{ (b^* T_n \Theta^Z)_+^\alpha, \sup_{t \ge 1} (b^* T_{n+t} \Theta^Z)_+^\alpha \} ] }
    {\sum_n \E [ \norm{T_n \Theta^Z}^\alpha ]}.
\end{align*}
For a nonnegative sequence $(a_n)_{n \in \ZZ}$ such that $\sum_n a_n < \infty$, one can show that
\[
  \sum_{n \in \ZZ} \min \biggl( a_n, \, \sup_{t \ge 1} a_{n+t} \biggr) = \sum_{n \in \ZZ} a_n - \sup_{n \in \ZZ} a_n.
\]
Combine the previous displays to arrive at \eqref{E:LinProc:ExtremalIndex:dual}.

By a similar argument, the extremal index in \eqref{E:ExtremalIndex:norm} of $(\norm{X_t})_{t \in \ZZ}$ is given by
\begin{equation}
\label{E:LinProc:ExtremalIndex:norm}
  \theta 
  = 1 - \E \biggl[ \min \biggl\{ 1, \sup_{t \ge 1} \norm{\Theta_t}_+^\alpha \biggr\} \biggr] 
  = \frac{\E [ \sup_n \norm{T_n \Theta^Z}^\alpha ]}{\sum_n \E [ \norm{T_n \Theta^Z}^\alpha ]}.
\end{equation}
\end{example}

\section{Examples of linear processes}
\label{S:Examples}

We conclude the paper with a number of special cases and examples of regularly varying linear processes. We strive to get explicit expressions for the spectral process and derived quantities.

\begin{example}[Isometries]
\label{ex:LinProc:Isometry}
Assume that the linear operators $T_i$ satisfy $\norm{T_i z} = \norm{T_i} \norm{z}$ for all $i \in \ZZ$ and all $z \in \Banach_1$, that is, the linear operators $T_i / \norm{T_i}$ are isometries, see Example~\ref{ex:series:Isometry}. The constants $c_n$ and $p_n$ in \eqref{E:series:c_n} and \eqref{E:series:p_n} respectively then are as given in equation~\eqref{E:series:Isometry:constants}. If $\norm{T_n} > 0$, the distribution $\kappa_n$ in \eqref{E:LinProc:kappa_n} on the sequence space $\Banach^\ZZ$ is the one of the random sequence
\[
  \biggl( \frac{T_{t+n}}{\norm{T_n}} \Theta^Z \biggr)_{t \in \ZZ}.
\]
By Proposition~\ref{P:LinProc}, the spectral process of $(X_t)_{t \in \ZZ}$ is given by
\begin{equation}
\label{E:LinProc:Isometry:spectral}
  (\Theta_t)_{t \in \ZZ} \eqd \biggl( \frac{T_{t+N}}{\norm{T_N}} \Theta^Z \biggr)_{t \in \ZZ}
\end{equation}
with $N$ an integer-valued random variable independent of $\Theta^Z$ and with distribution given by $\Pr(N = n) = p_n$ for $n \in \ZZ$. For instance, the extremal-index formula \eqref{E:LinProc:ExtremalIndex:norm} becomes
\[
  \theta = \frac{\sup_n \norm{T_n}^\alpha}{\sum_n \norm{T_n}^\alpha}.
\]
\end{example}

\begin{example}[Moving averages]
\label{ex:LinProc:MA}
Example~\ref{ex:LinProc:Isometry} applies to infinite-order moving averages of the type
\[
  X_t = \som_i \, Z_{t-i} \psi_i, \qquad t \in \ZZ,
\]
for random variables $Z_n$ and deterministic $\psi_i \in \Banach$ and also to moving averages of the type
\[
  X_t = \som_i \, a_i Z_{t-i}, \qquad t \in \ZZ,
\]
for constants $a_i \in \RR$ and $\Banach$-valued random elements $Z_n$; see Examples~\ref{ex:series:LinComb1} and~\ref{ex:series:LinComb2} respectively. In case $\Banach = \RR$, these two cases coincide and we recover the familiar framework of a linear time series in $\RR$ with heavy-tailed innovations: see for instance \cite{DR85}, \citet[Section~4.5]{RES87}, or \citet[Section~5.5.1]{EKM97}. We find that the spectral measure of $X_t$ is determined by
\begin{align*}
  \lim_{u \to \infty} \frac{\Pr(X_t > u)}{\Pr(|X_t| > u)} 
  &= \Pr(\Theta_0 = +1) \\
  &= \Pr[ \sign(a_N) \, \Theta_Z = +1 ] 
  = \frac{p \, \sum_n (a_n)_+^\alpha + q \, \sum_n (a_n)_-^\alpha}{\sum_n |a_n|^\alpha},
\end{align*}
where $p = \Pr(\Theta^Z = +1) = 1-q$ and $(x)_- = \max(-x, 0)$. More generally, the spectral process of $(X_t)_{t \in \ZZ}$ is given by
\begin{equation}
\label{E:LinProc:MA:spectral}
  (\Theta_t)_{t \in \ZZ} \eqd \biggl( \frac{a_{t+N}}{|a_N|} \Theta^Z \biggr)_{t \in \ZZ},
\end{equation}
with $N$ an integer-valued random variable independent of $\Theta^Z$ and with distribution given by $\Pr(N = n) = |a_n|^\alpha / \sum_k |a_k|^\alpha$ for $n \in \ZZ$.

By \eqref{E:TailDependence:dual} and \eqref{E:LinProc:MA:spectral}, the coefficient of tail dependence between $X_0$ and $X_h$ (lag $h \in \NN$) is equal to
\begin{multline*}
  \lim_{u \to \infty} \Pr(X_h > u \mid X_0 > u) \\
  = \frac{p \, \sum_n \min\{ (a_{n+h})_+^\alpha, \, (a_n)_+^\alpha \} + q \, \sum_n \min\{ (a_{n+h})_-^\alpha, \, (a_n)_-^\alpha \}}{p \, \sum_n (a_n)_+^\alpha + q \, \sum_n (a_n)_-^\alpha},
\end{multline*}
provided the denominator is positive. The formula in \eqref{E:LinProc:ExtremalIndex:dual} for the extremal index of the series $(X_t)_{t \in \ZZ}$ specializes to
\[
  \theta(+1) = \frac{p \, \sup_n (a_n)_+^\alpha + q \, \sup_n (a_n)_-^\alpha}{p \, \sum_n (a_n)_+^\alpha + q \, \sum_n (a_n)_-^\alpha},
\]
provided the denominator is positive. This expression coincides with the one in \citet[p.~415]{EKM97}.
\end{example}

\begin{example}[AR(1) processes]
\label{ex:AR}
When $\Banach_1 = \Banach_2 = \Banach$, a useful class of linear processes is given by the one of autoregressive processes of order one, AR(1) in short, defined by
\begin{equation}
\label{E:AR}
  X_t = TX_{t-1} + Z_t, \qquad t \in \ZZ.
\end{equation}
The innovations $Z_t$ are independent and identically distributed in $\Banach$ and $T : \Banach \to \Banach$ is a bounded linear operator. If there exists an integer $m \ge 1$ such that $\norm{T^m} < 1$, then the sequence $\norm{T^n}$ decays at a geometric rate as $n \to \infty$, so that \eqref{E:Resnick} is fulfilled for every $\delta > 0$. Hence, if the common distribution of the $Z_t$ is regularly varying with index $\alpha$, the AR(1) equation has a stationary solution given by
\[
  X_t = \sum_{n \ge 0} T^n Z_{t-n}, \qquad t \in \ZZ,
\]
where $T^0$ denotes the identity operator. Hence we are in the situation of Proposition~\ref{P:LinProc} with $T_n = T^n$ if $n \ge 0$ and $T_n = 0$ if $n < 0$. The tail of $\norm{X_t}$ satisfies
\[
  \lim_{x \to \infty} \frac{\Pr( \norm{X_t} > x )}{\Pr( \norm{Z_0} > x)} = \sum_{n \ge 0} \int_{\sphere} \norm{T^n \theta}^\alpha \, \lambda(d \theta)
\]
where $\lambda$ is the spectral measure of the innovations $Z_t$. In view of the term $n = 0$, the sum on the right-hand side cannot be smaller than unity. As a consequence, $(X_t)_{t \in \ZZ}$ is regularly varying with index $\alpha > 0$ and with spectral process as described in Proposition~\ref{P:LinProc}. Note that $p_n = 0$ for all $n < 0$ and that if $p_{n_0} = 0$ for some integer $n_0 \ge 1$, then necessarily $p_n = 0$ for all $n \ge n_0$.

The algorithm in \eqref{algorithm:kappa} for generating a random draw from the spectral process $(\Theta_t)_{t \in \ZZ}$ can be written as follows:
\begin{equation}
\label{algorithm:AR}
\text{
\begin{minipage}[t]{0.87\textwidth}
\begin{list}{}{\setlength{\leftmargin}{3ex}}
\item[1.] Draw a random nonnegative integer $N$ from $(p_n)_{n \ge 0}$.
\item[2.] Independently from $N$ and from each other, draw $\Theta^Z \sim \lambda$ and $U \sim \text{Uniform}(0, 1)$.
\item[3.] If $U \le \norm{T^N \Theta^Z}^\alpha / \norm{T^N}^\alpha$, then return
  \begin{align}
    \Theta_{-N} &= \frac{\Theta^Z}{\norm{T^N \Theta^Z}}, &
    \Theta_{-N+h} &= 
    \begin{cases}
      T^h \Theta_{-N} & \text{if $h > 0$,} \nonumber \\
      0 & \text{if $h < 0$.}
    \end{cases}
  \end{align}
\item[4.] Otherwise, go back to step 2.
\end{list}
\end{minipage}
}
\end{equation}
In step~3, the operator norm $\norm{T^N}$ may be replaced by the possibly smaller essential supremum of $\norm{T^N \Theta^Z}$, see \eqref{E:EssSup}. Note that the \emph{forward} spectral process satisfies $\Theta_t = T^t \Theta_0$ for $t \ge 0$.

In the special case that the operator $T$ satisfies $\norm{T^n \Theta^Z} = \norm{T^n}$ almost surely for all $n \in \ZZ$, then $p_n = \norm{T^n}^\alpha / \sum_{k \ge 0} \norm{T^k}^\alpha$ and the spectral process simplifies to
\begin{equation}
\label{E:AR:spectral:simple}
  (\Theta_t)_{t \in \ZZ} \eqd \biggl( \frac{T^{N+t}}{\norm{T^N}}  \, \Theta^Z \, \1_{\{t \ge - N\}} \biggr)_{t \in \ZZ}
\end{equation}
where $\Theta^Z \sim \lambda$ and $N \sim (p_n)_{n \ge 0}$ are independent.
\end{example}

\begin{example}[Sequence spaces]
\label{ex:Sequence}
Let $\Banach$ be equal to the real sequence space $\{ (x_n)_{n \ge 0} : \sum_n w_n |x_n| < \infty \}$, where $(w_n)_{n \ge 0}$ is a positive sequence satisfying $\sum_n w_n^\delta < \infty$ for some $\delta \in (0, 1)$; without loss of generality, assume $w_0 = 1$. Let $(\zeta_t)_{t \in \ZZ}$ be an i.i.d.\ sequence of random variables whose common distribution is regularly varying with index $\alpha$, where $\alpha > \delta$. For $n \in \ZZ$, let $I_n : \RR \to \Banach$ be the embedding $I_n z = z \, e_n$, with $e_n = (\delta_{n,k})_{k \ge 0}$ the $n$th unit vector in $\Banach$. Put
\begin{equation}
\label{E:Sequence}
  X_t = (\zeta_t, \zeta_{t-1}, \zeta_{t-2}, \ldots) = \sum_{n \ge 0} I_n \zeta_{t-n}, \qquad t \in \ZZ.
\end{equation}
Example~\ref{ex:LinProc:Isometry} applies with $\norm{I_n} = \norm{e_n} = w_n$. Alternatively, let $S : \Banach \to \Banach$ be the shift operator defined by
\begin{equation*}
\label{E:shift}
  S(x_0, x_1, x_2, \ldots) = (0, x_1, x_2, \ldots).
\end{equation*}
Since $I_n = S^n \circ I_0$, we also have the AR(1) representation
\begin{equation*}
\label{E:Sequence:AR}
  X_t = S X_{t-1} + Z_t = \som_{n \ge 0} \, S^n \, Z_{t-n}, \qquad t \in \ZZ,
\end{equation*}
with $Z_t = I_0 \zeta_t = (\zeta_t, 0, 0, \ldots)$. By an application of Proposition~\ref{P:linear:rv}, the random elements $Z_t$ are regularly varying in $\Banach$ with index $\alpha$ and spectral measure $\lambda_Z = \law(\Theta^\zeta, 0, 0, \ldots)$, where $\Theta^\zeta$ is a random variable in $\{-1, +1\}$ with distribution equal to the spectral measure of $\zeta_t$; see also Example~\ref{ex:Isometry}. However, if $w_{n+m} / w_n \to 1$ as $n \to \infty$ for every positive integer $m$, then $\norm{S^m} = 1$ for every such $m$ and condition~\eqref{E:Resnick} does not apply. Still, as the distribution of $Z_t$ is concentrated on the closed linear subspace $\Banach_0 = \operatorname{Im} I_0 = \{(x, 0, 0, \ldots) : x \in \RR\}$, condition~\eqref{E:Resnick:0} is verified with $\norm{S^m}_0 = w_m$ (recall $w_0 = 1$). We find that $(X_t)_{t \in \RR}$ is regularly varying with index $\alpha$, and by \eqref{E:AR:spectral:simple} its spectral process is given by
\begin{equation}
\label{E:Sequence:spectral}
  (\Theta_t)_{t \in \ZZ} \eqd \biggl( \frac{I_{N+t} \Theta^\zeta}{w_N} \, \1_{\{t \ge -N\}} \biggr)_{t \in \ZZ},
\end{equation}
with $\Theta^\zeta$ as above, independent from the positive-integer valued random variable $N$ with distribution $\Pr(N = n) = p_n = w_n^\alpha / \sum_{k \ge 0} w_k^\alpha$.

If $(a_n)_{n \ge 0}$ is a real sequence such that $\sup_n |a_n| / w_n < \infty$, then we can define a bounded linear functional $A$ on $\Banach$ by $Ax = \sum_n a_n x_n$. Applying this functional to the sequence $X_t$ in \eqref{E:Sequence:AR} yields
\[
  A X_t = \sum_{i \ge 0} a_i \zeta_{t-i},
\]
which, up to a change in notation, is exactly the real-valued linear process in Example~\ref{ex:LinProc:MA}. The results in that example can now also be recovered through Proposition~\ref{P:linear:series}. We leave the details for the reader.
\end{example}




\appendix

\section{Convergence of measures}
\label{A:aux}

\begin{lemma}
\label{L:CMT}
Assume $\mu_n \to \mu$ in $\Mzero(\Banach)$ as $n \to \infty$ and let $f : \Bzero \to \RR$ be bounded, measurable, and vanish on $B_{0,u}$ for some $u > 0$. Let $D$ be the discontinuity set of $f$. If $\mu(D) = 0$, then $\int f \, d\mu_n \to \int f \, d\mu$ as $n \to \infty$.
\end{lemma}

\begin{proof}
Let $r \in (0, u)$ be such that $\mu(\boundary B_{0,r}) = 0$. Let $\mu_n^{(r)}$ and $\mu_n^{(r)}$ denote the restrictions of $\mu_n$ and $\mu$ to $\Banach \setminus B_{0,r}$, respectively. By (the proof of) Theorem~2.2 in \cite{HL06}, we have $\mu_n^{(r)} \to \mu^{(r)}$ weakly. By the continuous mapping theorem for weak convergence of finite measures,
\[
  \int_\Bzero f \, d\mu_n 
  = \int_{\Banach \setminus B_{0,r}} f \, d\mu_n^{(r)}
  \to \int_{\Banach \setminus B_{0,r}} f \, d\mu^{(r)} 
  = \int_\Bzero f \, d\mu \qquad (n \to \infty). \qedhere
\]
\end{proof}


\begin{lemma}
\label{L:product}
Let $(S,d)$ be a separable metric space. Let $(X_n, Y_n)$ and $(X, Y)$ be random elements in $\RR \times S$. Then $(X_n, Y_n) \dto (X, Y)$ if and only if
\begin{equation}
\label{E:product}
  \E [ \1(X_n \le x) \, g(Y_n) ] \to \E [ \1(X \le x) \, g(Y) ] \qquad (n \to \infty)
\end{equation}
for every continuity point $x \in \RR$ of $X$ and every bounded and continuous function $g : S \to \RR$.
\end{lemma}

\begin{proof}
The `only if' part is a special case of the continuous mapping theorem. So assume \eqref{E:product} holds. Taking $g \equiv 1$ yields $X_n \dto X$. Taking $x$ arbitrarily large so that $\Pr(X > x)$ is arbitrarily small yields $Y_n \dto Y$. As a consequence, the sequence $(X_n, Y_n)$ is tight. It remains to show that the joint distribution of $(X, Y)$ is determined by expectations as in the right-hand side \eqref{E:product}. By Lemma~1.4.2 in \citet{VDV96}, the joint distribution of $(X, Y)$ is determined by expectations of the form $\E [ f(X) \, g(Y) ]$ with $f : \RR \to \RR$ and $g : S \to \RR$ nonnegative, Lipschitz continuous, and bounded. It then suffices to write $f$ as the limit of an increasing sequence of step functions whose jump locations are continuity points of $X$.
\end{proof}


\section{Tails of random series}
\label{A:tails}

The following result extends Lemma~4.24 in \cite{RES87}.

\begin{proposition}
\label{P:sum:inf}
Let $(Z_i)_{i \in \Z}$ be a sequence of independent and identically distributed random variables taking values in a Banach space $\B_1$. Let $\B_2$ be another Banach space and let $T_i : \B_1 \to \B_2$, $i \in \Z$, be bounded linear operators. Put $V(x) = \Pr(\norm{Z_i} > x)$. If 
\begin{itemize}
\item[(i)] $V \in \RV_{-\alpha}$ for some $\alpha > 0$,
\item[(ii)] $\lim_{x \to \infty} \Pr( \norm{T_i Z_i} > x ) / V(x) = a_i \in [0, \infty)$ exists for every $i \in \Z$,
\item[(iii)] $\som_i \norm{T_i}^\delta < \infty$ for some $0 < \delta < \min(\alpha, 1)$,
\end{itemize}
then the series $\sum_i T_i Z_i$ is almost surely absolutely convergent and
\begin{multline}
\label{E:sum:inf}
  \lim_{x \to \infty} \frac{1}{V(x)} \E \bigl| \1 ( \norm{\som_i T_i Z_i} > x ) - \som_i \1 ( \norm{T_i Z_i} > x ) \bigr| \\
  = \lim_{x \to \infty} \frac{1}{V(x)} \E \bigl| \1 ( \som_i \norm{T_i Z_i} > x ) - \som_i \1 ( \norm{T_i Z_i} > x ) \bigr|
  = 0.
\end{multline}
As a consequence,
\begin{multline}
\label{E:sum:inf:tail}
  \lim_{x \to \infty} \frac{\Pr ( \norm{\som_i T_i Z_i} > x )}{V(x)} 
  = \lim_{x \to \infty} \frac{\Pr ( \som_i \norm{T_i Z_i} > x )}{V(x)} \\
  = \lim_{x \to \infty} \frac{\som_i \Pr ( \norm{T_i Z_i} > x )}{V(x)} = \som_i a_i < \infty.
\end{multline}
\end{proposition}

The meaning of Proposition~\ref{P:sum:inf} is that asymptotically, the series $\sum_i T_i Z_i$ is large in norm if and only if one of the summands $T_i Z_i$ is large in norm, in which case $\norm{\sum_i T_i Z_i}$ and $\sum_i \norm{T_i X_i}$ are asymptotically equivalent---the value of the series is determined by a single large `shock'. The proof of Proposition~\ref{P:sum:inf} is based on a number of lemmas.

\begin{lemma}
\label{L:sum:2}
Let $X$ and $Y$ be independent random elements of a Banach space $\B$. Suppose that there exists $V \in \RV_{-\alpha}$ with $\alpha > 0$ and nonnegative numbers $a, b$ such that
\begin{align}
\label{E:ab}
	\lim_{x \to \infty} \frac{\Pr( \norm{X} > x )}{V(x)} &= a, &
	\lim_{x \to \infty} \frac{\Pr( \norm{Y} > x )}{V(x)} &= b.
\end{align}
Then as $x \to \infty$,
\begin{multline}
\label{E:sum:2}
  \lim_{x \to \infty} \frac{1}{V(x)} \E \bigl| \1( \norm{X + Y} > x ) - \1( \norm{X} > x ) - \1( \norm{Y} > x ) \bigr| \\
  = \lim_{x \to \infty} \frac{1}{V(x)} \E \bigl| \1( \norm{X} + \norm{Y} > x ) - \1( \norm{X} > x ) - \1( \norm{Y} > x ) \bigr|
  = 0.
\end{multline}
\end{lemma}

\begin{proof}
Use the formula $|r| = (r)_+ + (-r)_+$ for real $r$ to decompose the absolute values in \eqref{E:sum:2}.

Fix $0 < \eps < 1$. For positive numbers $r, s, x$, if $r + s > x$, then necessarily $r \vee s > (1 - \eps) x$ or $r \wedge s > \eps x$. It follows that
\begin{align*}
	\1 ( \norm{X + Y} > x)
	&\le \1( \norm{X} + \norm{Y} > x ) \\
	&\le \1 \{ \norm{X} \vee \norm{Y} > (1-\eps) x \} 
	+ \1 ( \norm{X} \wedge \norm{Y} > \eps x ) \\
	&\le \1 \{ \norm{X} > (1 - \eps) x \}
	+ \1 \{ \norm{Y} > (1 - \eps) x \} \\
	& \qquad \mbox{} + \1 ( \norm{X} > \eps x ) \1 ( \norm{Y} > \eps x ).
\end{align*}
We find that
\begin{multline*}
  \1( \norm{X + Y} > x ) - \1( \norm{X} > x ) - \1( \norm{Y} > x ) \\
  \shoveleft \le \1( \norm{X} + \norm{Y} > x ) - \1( \norm{X} > x ) - \1( \norm{Y} > x ) \\
  \shoveleft \le \1 \{ (1-\eps) x < \norm{X} \le x \} + \1 \{ (1-\eps) x < \norm{Y} \le x \} \\
  + \1 ( \norm{X} > \eps x ) \1 ( \norm{Y} > \eps x ).
\end{multline*}
By regular variation of $V$ and by \eqref{E:ab},
\begin{multline}
\label{E:sum:2:+}
  \limsup_{x \to \infty} \frac{1}{V(x)} \E \bigl[ \bigl( \1( \norm{X + Y} > x ) - \1( \norm{X} > x ) - \1( \norm{Y} > x ) \bigr)_+ \bigr] \\
  \le \limsup_{x \to \infty} \frac{1}{V(x)} \E \bigl[ \bigl( \1( \norm{X} + {Y} > x ) - \1( \norm{X} > x ) - \1( \norm{Y} > x ) \bigr)_+ \bigr] \\
  \le (a + b) \bigl( (1 - \eps)^{-\alpha} - 1 \bigr).
\end{multline}
Since $\eps > 0$ was arbitrary, both limits superior in \eqref{E:sum:2:+} are actually zero.

Using the inequality $\norm{v + w} \ge |\norm{v} - \norm{w}|$ for $v, w \in \B$, we have
\begin{align*}
	\1( \norm{X + Y} > x)
	&\ge \1 (|\norm{X} - \norm{Y}| > x) \\
	&= \1 ( \norm{X} > x + \norm{Y} ) + \1 ( \norm{Y} > x + \norm{X} ).
\end{align*}
It follows that
\begin{multline*}
  \1( \norm{X} > x ) + \1( \norm{Y} > y ) - \1( \norm{X} + \norm{Y} > x ) \\
  \le \1( \norm{X} > x ) + \1( \norm{Y} > y ) - \1( \norm{X + Y} > x ) \\
  \le \1 ( x < \norm{X} \le x + \norm{Y} ) + \1 ( x < \norm{Y} \le x + \norm{X} ).
\end{multline*}
Fix $\eps > 0$. We have
\begin{multline*}
	\1 ( x < \norm{X} \le x + \norm{Y} ) \\
	\le \1 \{ x < \norm{X} \le (1+\eps) x \}
	+ \1 ( \norm{X} > x ) \1 ( \norm{Y} > \eps x ).
\end{multline*}
By regular variation of $V$ and \eqref{E:ab},
\[
	\limsup_{x \to \infty}
	\frac{\Pr( x < \norm{X} \le x + \norm{Y} )}{V(x)}
	\le a \bigl( 1 - (1+\eps)^{-\alpha} \bigr).
\]
By symmetry, we obtain
\begin{multline}
\label{E:sum:2:-}
  \limsup_{x \to \infty} \frac{1}{V(x)} \E \bigl[ \bigl( \1( \norm{X} > x ) + \1( \norm{Y} > y ) - \1( \norm{X} + \norm{Y} > x ) \bigr)_+ \bigr] \\
  \le \limsup_{x \to \infty} \frac{1}{V(x)} \E \bigl[ \bigl( \1( \norm{X} > x ) + \1( \norm{Y} > y ) - \1( \norm{X + Y} > x ) \bigr)_+ \bigr] \\
  \le (a + b) \bigl( 1 - (1+\eps)^{-\alpha} \bigr).
\end{multline}
Since $\eps > 0$ was arbitrary, both limits superior in \eqref{E:sum:2:-} are actually zero. Finally, combine \eqref{E:sum:2:+} and \eqref{E:sum:2:-} to arrive at \eqref{E:sum:2}.
\end{proof}

\begin{lemma}
\label{L:sum:n}
Let $X_1, \ldots, X_n$ be independent random elements of a Banach space $\B$. Suppose that there exists $V \in \RV_{-\alpha}$ with $\alpha > 0$ and nonnegative numbers $a_1, \ldots, a_n$ such that
\begin{equation}
\label{E:sum:n:ai}
	\lim_{x \to \infty} \frac{\Pr( \norm{X_i} > x )}{V(x)} = a_i, \qquad i \in \{1, \ldots, n\}.
\end{equation}
Then
\begin{multline}
\label{E:sum:n}
  \lim_{x \to \infty} \frac{1}{V(x)} \E \bigl| \1 \bigl( \norm{ \som_{i=1}^n X_i } > x \bigr) - \som_{i=1}^n \1 ( \norm{X_i} > x ) \bigr| \\
  = \lim_{x \to \infty} \frac{1}{V(x)} \E \bigl| \1 \bigl( \som_{i=1}^n \norm{X_i} > x \bigr) - \som_{i=1}^n \1 ( \norm{X_i} > x ) \bigr| = 0
\end{multline}
In particular,
\begin{equation}
\label{E:sum:n:tail}
  \lim_{x \to \infty} \frac{\Pr( \norm{ \som_{i=1}^n X_i } > x )}{V(x)}
  = \lim_{x \to \infty} \frac{\Pr( \som_{i=1}^n \norm{X_i} > x )}{V(x)}
  = \sum_{i=1}^n a_i.
\end{equation}
\end{lemma}

\begin{proof}
The proof is by induction on $n$, using Lemma~\ref{L:sum:2}.
\end{proof}

\begin{proof}[Proof of Proposition~\ref{P:sum:inf}]
Since $\sum_i \norm{T_i Z_i} \le \sum_i \norm{T_i} \norm{Z_i}$, conditions (i) and (iii) together with the argument on p.~225 in \cite{RES87} imply absolute convergence of the series $\sum_i T_i Z_i$ almost surely.

Write $X_i = T_i Z_i$. In order to prove \eqref{E:sum:inf}, let $m$ be a positive integer. Write
\begin{multline}
\label{E:sum:inf:decomp}
	\bigl| \1 ( \norm{\som_i X_i} > x ) - \som_i \1 ( \norm{X_i} > x ) \bigr| \\
	\le \bigl| \1 (\norm{\som_i X_i} > x) - \1 (\norm{\som_{|i| \le m} X_i} > x) \bigr| \\
	\shoveright{+ \bigl| \1 (\norm{\som_{|i| \le m} X_i} > x) 
	- \som_{|i| \le m} \1 ( \norm{X_i} > x) \bigr|} \\
	+ \som_{|i| > m} \1 (\norm{X_i} > x).
\end{multline}
The three terms on the right-hand side are treated separately.

Consider the first term on the right-hand side of equation~\eqref{E:sum:inf:decomp}. Let $X = \som_{|i| \le m} X_i$ and $Y = \som_{|i| > m} X_i$. Using the identity $|\1_A - \1_B| = \1_{A \symdif B}$, it is not difficult to see that for $0 < \eps < 1$,
\begin{multline*}
	\bigl| \1( \norm{X + Y} > x ) - \1( \norm{X} > x ) \bigr| \\
	\le \1 (\norm{Y} > \eps x) + \1\{(1-\eps)x < \norm{X} \le (1+\eps)x \}.
\end{multline*}
By Lemma~4.24 in \cite{RES87}, since $\norm{Y} \le \sum_{|i| > m} \norm{T_i} \norm{Z_i}$, 
\[
	\limsup_{y \to \infty} \frac{\Pr( \norm{Y} > y )}{V(y)}
	\le \sum_{|i| > m} \norm{T_i}^\alpha.
\]
By (ii) and \eqref{E:sum:n:tail},
\[
	\lim_{y \to \infty} \frac{\Pr( \norm{X} > y )}{V(y)}
	= \sum_{|i| \le m} a_i.
\]
By regular variation of $V$ and from $\norm{T_i Z_i} \le \norm{T_i} \norm{V_i}$, it follows easily that $a_i \le \norm{T_i}^\alpha$. In view of the three preceding displays,
we have
\begin{multline}
\label{E:sum:inf:decomp:A}
	\limsup_{x \to \infty} \frac{1}{V(x)} 
	\E \bigl| \1( \norm{X + Y} > x ) - \1( \norm{X} > x ) \bigr| \\
	\le \eps^{-\alpha} \sum_{|i| > m} \norm{T_i}^\alpha
	+ \bigl( (1-\eps)^{-\alpha} - (1+\eps)^{-\alpha} \bigr) \sum_{|i| \le m} a_i \\
	\le \eps^{-\alpha} \sum_{|i| > m} \norm{T_i}^\alpha
	+ \bigl( (1-\eps)^{-\alpha} - (1+\eps)^{-\alpha} \bigr) \sum_{i \in \Z} \norm{T_i}^\alpha.
\end{multline}

Next, for the second term on the right-hand side of \eqref{E:sum:inf:decomp}, Lemma~\ref{L:sum:n} implies immediately that, as $x \to \infty$,
\begin{equation}
\label{E:sum:inf:decomp:B}
	\E \bigl| \1 (\norm{\som_{|i| \le m} X_i} > x) - \som_{|i| \le m} \1 ( \norm{X_i} > x) \bigr| = o\bigl( V(x) \bigr).
\end{equation}

Finally, consider the third term on the right-hand side of \eqref{E:sum:inf:decomp}. Let $m$ be sufficiently large such that $\norm{T_i} \le 1$ whenever $|i| > m$. By Potter's theorem, there exists $x_0 > 0$ such that $V(c x) / V(x) \le 2 c^{-\delta}$ for every $x \ge x_0$ and $c \ge 1$. Hence
\begin{equation}
\label{E:Potter:TiZi}
	\frac{\Pr( \norm{X_i} > x )}{V(x)}
	\le \frac{\Pr( \norm{T_i} \norm{Z_i} > x )}{V(x)} 
	\le 2 \norm{T_i}^\delta, \qquad x \ge x_0, \, |i| > m.
\end{equation}
As a consequence,
\begin{equation}
\label{E:sum:inf:decomp:C}
	\limsup_{x \to \infty} \frac{1}{V(x)}
	\sum_{|i| > m} \Pr( \norm{X_i} > x ) 
	\le 2 \sum_{|i| > m} \norm{T_i}^\delta.
\end{equation}

Taking expectations in \eqref{E:sum:inf:decomp} and inserting the upper bounds in \eqref{E:sum:inf:decomp:A}, \eqref{E:sum:inf:decomp:B} and \eqref{E:sum:inf:decomp:C} yields the following inequality, valid for all $0 < \eps < 1$ and for all sufficiently large integer $m$:
\begin{multline*}
	\limsup_{x \to \infty} \frac{1}{V(x)}
	\E\bigl| \1 ( \norm{\som_i X_i} > x ) - \som_i \1 ( \norm{X_i} > x ) \bigr| \\
	\le \eps^{-\alpha} \sum_{|i| > m} \norm{T_i}^\alpha 
	+ \bigl( (1-\eps)^{-\alpha} - (1+\eps)^{-\alpha} \bigr) \sum_{i \in \Z} \norm{T_i}^\alpha 
	+ 2 \sum_{|i| > m} \norm{T_i}^\delta.
\end{multline*}
First let $m \to \infty$ and then let $\eps \downarrow 0$ to arrive at \eqref{E:sum:inf} for $\norm{\sum_i T_i Z_i}$. The proof for $\sum_i \norm{T_i Z_i}$ is completely similar.
	
Equation~\eqref{E:sum:inf:tail} follows from \eqref{E:sum:inf} and the dominated convergence theorem, the use of which is justified by (iii) and \eqref{E:Potter:TiZi}.
\end{proof}


\end{document}